\documentclass[reqno]{amsart}  %LaTeX2e

\theoremstyle{plain}
\newtheorem{theorem}{Theorem}[section]
\newtheorem{lemma}[theorem]{Lemma}
\newtheorem{remark}[theorem]{Remark}
\newtheorem{proposition}[theorem]{Proposition}
\newtheorem{problem}[theorem]{Problem}

\newtheorem{corollary}[theorem]{Corollary}

\newtheorem{definition}[theorem]{Definition}

\numberwithin{equation}{section}

\newcommand{\bU}{{\mathbf U}}
\newcommand{\cB}{{\mathcal B}}
\newcommand{\cC}{{\mathcal C}}

\newcommand{\cH}{{\mathcal H}}

\newcommand{\cL}{{\mathcal L}}
\newcommand{\cM}{{\mathcal M}}
\newcommand{\cP}{{\mathcal P}}
\newcommand{\cR}{{\mathcal R}}
\newcommand{\cK}{{\mathcal K}}
\newcommand{\cS}{{\mathcal S}}

\newcommand{\cU}{{\mathcal U}}
\newcommand{\cX}{{\mathcal X}}
\newcommand{\cY}{{\mathcal Y}}

\newcommand{\bu}{{\mathbf u}}
\newcommand{\bv}{{\mathbf v}}
\newcommand{\bphi}{{\boldsymbol \phi}}
\newcommand{\bo}{{\boldsymbol 0}}
\newcommand{\bx}{{\mathbf x}}
\newcommand{\bw}{{\mathbf w}}

\newcommand{\R}{\text{\rm Re }}
\newcommand{\I}{\text{\rm Im }}

\begin{document}

\title[Convexity analysis and Schur class]{Convexity analysis and matrix-valued
Schur class over finitely connected planar domains}
\author[J.A.~Ball]{Joseph A.~Ball}
\address{Department of Mathematics,
Virginia Tech,
Blacksburg, VA 24061-0123, USA}
\email{joball@math.vt.edu}
\author[M.D.~Guerra-Huam\'an]{Mois\'es D.~Guerra Huam\'an}
\address{Department of Mathematics,
Virginia Tech,
Blacksburg, VA USA}
\email{moisesgg@math.vt.edu}

\begin{abstract} We identify the set of extreme points and apply
    Choquet theory to a normalized matrix-measure ball subject to
    finitely many linear side constraints.  As an application we
    obtain integral representation formulas for the Herglotz class of
    matrix-valued functions on a finitely-connected planar domain and
    associated continuous Agler decompositions for the matrix-valued
    Schur class over the domain.  The results give some additional
    insight into the negative answer to the spectral set problem over
    such domains recently obtained by Agler-Harland-Raphael and 
    Dritschel-McCullough.
\end{abstract}


\maketitle

\section{Introduction}   \label{S:intro}

We define the classical Schur class (operator-valued version) 
$\cS(\cU, \cY)$ to be 
the class of holomorphic functions $z \mapsto S(z)$ from the unit disk ${\mathbb D}$ 
into contraction operators  between two coefficient Hilbert spaces 
$\cU$, $\cY$.  This class has been an object of much study and a 
source of much inspiration over the last several decades due to its 
central role in a number of applications but also due to the rich 
commingling of function theory, operator theory and engineering system 
theory ideas in the description of its structure.  Let us mention 
several equivalent characterizations/points of view toward the Schur 
class:  (1) the operator $M_{S}$ of multiplication by $S$ defines a 
contraction operator on the Hardy space over ${\mathbb D}$, (2) the 
de Branges-Rovnyak kernel 
\begin{equation}   \label{deBRker}
K_{S}(z,w) = \frac{I - S(z) S(w)^{*}}{1 - z \overline{w}}
\end{equation}
is a positive kernel over ${\mathbb D}$, (3) $S$ can 
be realized as the transfer function of a conservative discrete-time 
input/state/output linear system:  
\begin{equation}   \label{realization}
    S(z) = D + z C (I - zA)^{-1} B 
 \text{ with }\bU =  \begin{bmatrix} A & B \\ C & D 
\end{bmatrix} \colon \begin{bmatrix} \cX \\ \cU \end{bmatrix} \to 
\begin{bmatrix} \cX \\ \cY \end{bmatrix} \text{ unitary}.  
\end{equation}
A major step forward in 
developing an analogous theory in several-variable settings was made 
by Agler \cite{Agler90} where what we now call the {\em Schur-Agler 
class} over the polydisk was introduced.  This class is defined as 
the class of operator-valued functions on the polydisk ${\mathbb 
D}^{d} = \{ z = (z_{1}, \dots, z_{d} ) \colon |z_{k}| < 1 \text{ for } 
k = 1, \dots, d\}$ such that not only $\|S(z)\| \le 1$ for each $z 
\in {\mathbb D}^{d}$ but also $\|S(T) \| \le 1$ for all commutative 
$d$-tuples $T = (T_{1}, \dots, T_{d})$ of operators on some Hilbert 
space $\cK$, where e.g.~$S(T)$ can be defined as
$$
S(T) = \sum_{n \in {\mathbb Z}^{d}_{+}} S_{n} \otimes T^{n} \text{ if } 
S(z) = \sum_{n \in {\mathbb Z}^{d}_{+}} S_{n} z^{n}.
$$
A new feature for this class is the analogue of positivity of the 
kernel \eqref{deBRker}: rather than a characterization in terms of the 
positivity of a single kernel, the characterization is in terms of 
being able to solve for $d$ positive kernels $K_{1}, \dots, K_{d}$ on 
the polydisk so that the so-called {\em Agler decomposition} holds:
\begin{equation}   \label{Aglerdecom}
  I - S(z) S(w)^{*} = \sum_{k=1}^{d} (1 - z_{k} \overline{w_{k}}) 
  K_{k}(z,w).
\end{equation}
Also the realization formula \eqref{realization} for the 
multivariable Schur-Agler class takes the form
\begin{equation}   \label{Agler-real}
S(z) =  D + C (I - Z(z) A)^{-1} Z(z) B \text{ with } \bU  = 
\begin{bmatrix} A & B \\ C & D \end{bmatrix} \colon \begin{bmatrix} 
\cX \\ \cU \end{bmatrix} \to \begin{bmatrix} \cX \\ \cU \end{bmatrix}
\text{ unitary }
\end{equation}
with $Z(z) = \sum_{k=1}^{d} z_{k} P_{k}$ where $P_{1}, \dots, P_{d}$ 
form a spectral family of projection operators ($P_{i} = P_{i}^{*}$,
$P_{i} P_{j} = \delta_{i,j} I_{\cX}$,  $\sum_{k=1}^{d} P_{k} = 
I_{\cX}$) on the state space $\cX$.

One of our main motivations for the present paper was to further 
develop the understanding of the Schur class $\cS(\cR)$ over a 
bounded finitely-connected planar domain $\cR$.  Here $\cS(\cR)$ denotes the 
class of holomorphic functions mapping the planar domain $\cR$ into 
the unit disk.  We shall use the notation $\cS^{N}(\cR)$ for the 
class of holomorphic functions mapping $\cR$ into contractive $N \times N$
matrices.  In the course of constructing a counterexample to the 
spectral set question over $\cR$, Dritschel and McCullough \cite{DM} 
obtained a continuous analogue of the Agler decomposition 
\eqref{Aglerdecom} for the scalar-valued Schur class $\cS(\cR)$ over 
$\cR$.  Specifically,  let $\partial_{0}, \dots, \partial_{m}$ denote 
the $m+1$ connected components of the boundary $\partial \cR$ of 
$\cR$ (with $\partial_{0}$ the component which is the boundary of the 
unbounded component of the complement of $\cR$ in the complex plane)
and let ${\mathbb T}_{\cR}$ denote the Cartesian product ${\mathbb 
T}_{\cR} = \partial_{0} \times \cdots \times \partial_{m}$. The 
coordinate functions $z_{1}, \dots, z_{d}$ appearing in the Agler 
decomposition \eqref{Aglerdecom} must be replaced by a continuum $\{ 
s_{\bx}(z) \colon  \bx \in {\mathbb T}_{\cR}\}$ of single-valued 
inner functions on $\cR$ (i.e., holomorphic in $\cR$ with modulus-1 
values on $\partial \cR$), each with  $m$ zeros in $\cR$
(the minimal number possible for a nonconstant single-valued inner 
function),
indexed by the so-called $\cR$-torus ${\mathbb T}_{\cR}$.  Then the 
result of Dritschel-McCullough can be formulated as follows:

\begin{theorem}  \label{T:DMintro}
Given any $s \in \cS(\cR)$, then there is a family $k_{\bx}(z,w)$ of 
positive kernels on $\cR$, indexed by ${\mathbb T}_{\cR}$ and 
measurable on ${\mathbb T}_{\cR}$ for each $(z,w) \in \cR \times 
\cR$, so that 
\begin{equation}   \label{DMdecom}
1 - s(z) \overline{s(w)} = 
\int_{{\mathbb T}_{\cR}} (1 - s_{\bx}(z) 
\overline{s_{\bx}(w)}) k_{\bx}(z,w) \, {\tt d} \nu(\bx).
\end{equation}
\end{theorem}

There is also obtained in \cite{DM} a more elaborate version of the 
realization formula \eqref{realization} or \eqref{Agler-real} for the Schur class 
$\cS(\cR)$ which we do not go into here.  We also mention that these 
techniques actually lead to interpolation theorems for the various 
Schur classes:  if the function $S$ is initially given only on some 
(possibly finite) subset of its domain (${\mathbb D}$, ${\mathbb 
D}^{d}$, or $\cR$), then a necessary and sufficient condition for 
there to be an extension to the whole domain which is in the 
appropriate Schur class is that the decomposition \eqref{deBRker}, 
\eqref{Aglerdecom}, \eqref{DMdecom} hold for $z,w$ in the subset.
A dual version of the interpolation result for the class $\cS(\cR)$, 
whereby one tests the positivity of each kernel from a collection of kernels $\{ (1 
- s(z) \overline{s(w)}) k^{(\alpha)}(z,w)\}$ (where $k^{(\alpha)}(z,w)$ 
is a collection of Szeg\H{o}-type kernels indexed by $\alpha$ 
from the $m$-torus ${\mathbb T}^{m}$), was obtained earlier by 
Abrahamse \cite{Abrahamse}.

While Dritschel-McCullough indicated some results for the 
matrix-valued Schur class over $\cR$ on their way to constructing a 
counterexample to the spectral set question over $\cR$, the analogue 
of \eqref{DMdecom} for the matrix-valued case was left rather 
mysterious.  In general, extensions of scalar-valued results to 
the matrix-valued case for the Schur class over a planar domain $\cR$ 
have led to surprises:  it is known for example that the Abrahamse interpolation 
result does not extend to the matrix-valued case without the addition 
of additional matrix-valued kernels $k^{({\boldsymbol \alpha})}$ 
(see \cite{McCIEOT2001, McCP,  BBtH, DPRS}).

One of the main motivations of the present paper was to 
find an appropriate analogue of the Dritschel-McCullough 
decomposition \eqref{DMdecom} for the matrix-valued setting; such an 
analogue appears as Theorem \ref{T:Schur-def} below.
The basic idea 
in \cite{DM} for getting the decomposition \eqref{DMdecom} 
is to apply a linear-fractional change of variable on the range of 
the function to covert the problem to a problem concerning the 
Herglotz class over $\cR$  (holomorphic functions on $\cR$ with 
positive real part).  When this class is normalized by the condition 
that all such functions $f$ have the value 1 at some fixed 
point   
$t_{0} \in \cR$, it becomes a compact convex set.  Once one 
identifies the extreme points for this class, Choquet theory (see 
e.g.~\cite{Phelps} for a thorough account) can be applied to obtain 
an integral representation for a given Herglotz-class function $f$ in 
terms of the extreme points $f_{\bx}$.  The Cayley transforms of 
these extreme points for the Herglotz class turn out to be unimodular 
scalar multiples of the 
inner functions with exactly $m$ zeros appearing in the decomposition 
\eqref{DMdecom}:  $s_{\bx}(z) = \frac{f_{\bx}(z) -1}{f_{\bx}(z) + 1}$.
Explicit identification of the extreme points $f_{\bx}$ involves some 
clever function theory (see \cite{Grunsky, AHR, DM, Pick}).
The starting point is the  Poisson-kernel representation for 
positive harmonic functions. This leads to a one-to-one correspondence 
between normalized Herglotz functions on $\cR$ and probability 
measures on $\partial \cR$ which satisfy $m$ additional linear 
constraints ($m$ equal to the number of holes in $\cR$).  In this way 
extremal normalized Herglotz functions correspond to probability 
measures which are extremal in this set of linearly-constrained 
probability measures.  The problem of characterizing the extreme 
points of such a set of probability measures can be formulated in the 
setting of an abstract Borel measure space $X$ (in place of $\partial 
\cR$).  We study this general problem and give a geometric 
characterization of the extreme points in terms of $0$ being in the 
interior of the convex hull of a given collection of vectors in 
${\mathbb R}^{m}$, putting the results of Dritschel-Pickering in \cite{DP}
into a broader context. 

The extension of these ideas to the matrix-valued setting leads to new 
issues to be understood. 
Each $N \times N$-matrix valued Herglotz functions normalized to be 
the identity $I_{N}$ at the fixed point $t_{0} \in \cR$ corresponds 
to a quantum probability measure, i.e., a positive matrix-valued 
measure  $\mu$ on $\cR$ with total mass $\mu(\partial \cR)$ equal to the 
identity matrix $I_{N}$, subject to $m$ linear side constraints 
(given by integration against $m$ continuous real-valued functions 
on $\partial \cR$).  This problem in turn can be considered more 
generally, where $\partial \cR$ is replaced by a general Borel space 
$X$.  The problem then is to characterize the set of extreme points 
of the compact convex cone of quantum probability measures subject to 
$m$ linear side constraints.  It turns out that 
the special case of this problem where there are no side constraints has 
been analyzed and solved by Arveson \cite{Arv}:  extremal measures 
$\mu$ are characterized by the condition that $\mu = \sum_{k=1}^{n} 
W_{k} \delta_{x_{k}}$ (where $\delta_{x_{k}}$ is the scalar unit 
point-mass measure at the point $x_{k}$ and $W_{k} \ge 0$ is a matrix 
weight) where the family of subspaces $\{\operatorname{Ran} W_{k} \colon 1 \le k \le 
n\}$ should satisfy a condition called {\em weak independence} which, 
as suggested by the terminology, is somewhat weaker than the standard 
linear algebra notion of linear independence of subspaces (i.e., any 
collection of nonzero vectors $x_{1}, \dots, x_{d}$ with $x_{k} \in 
W_{k}$ should be a linear independent set of vectors in the standard 
sense).  We obtain an extension of Arveson's result to the constrained case
which has a geometric interpretation analogous to that in \cite{DP} 
for the scalar-case, namely:  the $0$ vector must be in the interior 
of the $C^{*}$-convex hull of a given set of matrix-tuples (see 
Remark \ref{R:noncomcon} below), thereby providing links with the general area of 
noncommutative convexity as in \cite{FM-PAMS, FM, Gregg}.
Finally we apply this general result on extreme points
to obtain a characterization (although not quite as explicit as 
in the scalar-valued case) of the extreme points of the normalized matrix-valued 
Herglotz class over a planar domain $\cR$. 

We do not treat here the transfer-function realization  and 
interpolation theory for the matrix-valued Schur class 
$\cS^{N}(\cR)$.  Such results can be obtained as part of a general 
theory of matrix-valued Schur class associated with a collection of 
matrix-valued test functions.  We address this topic beyond what 
already appears in \cite{Mo} in a separate report \cite{BG}.

The paper is organized as follows.  Following this Introduction, 
Section \ref{S:prelim} sets notation and reviews results from 
convexity theory (in particular, the Choquet-Bishop-de Leeuw theory 
on integral representations for points of a compact, convex set) 
which will be needed in the sequel.  Section \ref{S:extreme} considers the 
extreme-point problem for a linearly-constrained normalized set of 
positive matrix measures in the general measure-theory framework.  
Section \ref{S:Herglotz} introduces the function-theory setting and applies the 
theory of Section \ref{S:extreme} to obtain characterizations of 
extreme points and integral representations for normalized 
matrix-valued Herglotz-class functions over a finitely-connected 
planar domain $\cR$.  Section \ref{S:Schur} applies the 
linear-fractional change of variable to convert the results concerning 
Herglotz-class functions to results concerning Schur-class functions 
over $\cR$.  The final Section \ref{S:spectral} presents connections 
with the spectral set question over a region $\cR$: it turns out that 
the recent negative solution of the spectral set question can be 
partially explained by the lack of a simple transition formula from 
the extreme points for the scalar-valued normalized Herglotz class to 
the extreme points for the matrix-valued normalized Herglotz class 
over $\cR$ (see Corollary \ref{C:special} below).

Preliminary versions of many of the results described appear already 
in the Virginia Tech dissertation of the second author \cite{Mo}.
Finally we would like to thank David Sherman of the University of
Virginia for several helpful discussions on various topics developed
in this paper.

\section{General convexity theory}  \label{S:prelim}

A subset $\cC$ of a real linear space $C$ is said to be {\em convex} if, 
given any collection of vectors $\bu_{1}, \dots, \bu_{n}$ in 
$\cC$ and a collection of nonnegative real numbers $\lambda_{1}, 
\dots, \lambda_{n}$ with $\lambda_{1} + \cdots + 
\lambda_{n} = 1$, it happens that the convex linear combination
$ \sum_{i=1}^{n} \lambda_{1} \bu_{i}$ is again in $\cC$.
Given any subset $\cS$ of the linear space $E$, there is always a 
smallest subset of $E$ containing $\cS$, denoted as 
$\operatorname{conv} \cS$ (the {\em convex hull} of $\cS$).

A vector $\bv$ in the convex set $\cC$ is 
said to be an {\em extreme point} of $\cC$ if, whenever it is the 
case that $\bv = \lambda_{1} \bu_{1} + (1 - \lambda) \bu_{2}$ for a 
real $\lambda$ with $0 < \lambda < 1$ and $\bu_{1}$ and $\bu_{2}$ in 
$\cC$, it follows that $\bu_{1} = \bu_{2} = \bv$. The following characterization of 
extreme point is often easier to apply than the definition.

\begin{lemma} \label{L:convex}
      The point $\bv \in {\mathcal C}$ is an extreme
    point of the convex set ${\mathcal C}$ ($\bv \in \partial_{e}{\mathcal C}$) if and
    only if the following condition holds: whenever $\bu \in E$ is such
    that $\bv \pm \bu \in {\mathcal C}$, then $\bu = 0$.
\end{lemma}

\begin{proof}  Suppose $\bv$ is extreme and $\bv \pm \bu \in \cC$ or 
    some $\bu \in E$.  Since 
    $\bv$ is extreme, from the identity
    $$ \bv = \frac{1}{2}(\bv + \bu) + \frac{1}{2} (\bv - \bu)
    $$
 we see immediately that $\bu = 0$.
 
 For the converse it suffices to show the contrapositive:  {\em $\bv$ 
 not extreme $\Rightarrow$ there is a $\bu \ne 0$ in $E$ with $\bv \pm 
 \bu \in \cC$.}  If $\bv$ is not extreme, then we can find $\bv_{1}, 
 \bv_{2}$ in $\cC$ distinct from $\bu$ so that $\bv = \lambda \bv_{1} 
 + (1 - \lambda) \bv_{2}$.  We rearrange this as
 $$
 \lambda(\bv - \bv_{1}) = (1-\lambda) (\bv_{2} - \bv) = :\bu.
 $$
 Then 
 \begin{align*}
&   \bv + \bu = \bv + (1- \lambda) (\bv_{2} - \bv) = \lambda \bv + ( 1 
  - \lambda) \bv_{2} \in \cC, \\
 & \bv - \bu = \bv - \lambda (\bv - \bv_{1}) = (1 - \lambda) \bv + 
 \lambda \bv_{1} \in \cC
\end{align*}
from which we see that the vector $\bu$ has the needed property.
\end{proof}

In case the linear space $E$ carries a locally convex topology and 
the convex subset $\cC$ is compact in this topology, the well known 
theorem of Kre\u{\i}n-Milman (see e.g.\cite[page 75]{RudinFA}) 
asserts that $\cC$ is the closure of its set of extreme points 
$\partial_{e} \cC$.  There is a refinement of the Kre\u{\i}n-Milman 
theorem known generically as {\em Choquet theory}.  
In general let us say that a vector $\bv$ in the nonempty compact 
subset $X$ of the linear topological vector space $E$ is {\em 
represented} by the probability Borel measure $\nu$ on $X$ if it is 
the case that 
$$
  \ell(\bv) = \int_{X} \ell(\bu)\, {\tt d}\nu(\bu)
 $$
 for all continuous linear functionals $\ell \in E^{*}$.  A 
 consequence of the Hahn-Banach theorem then is that $\nu$ uniquely 
 determines the element $\bv \in E$.  The following 
theorem summarizes what we need from Choquet theory and is due mainly to
Choquet \cite{Choquet} and Bishop-de Leeuw \cite{BL}.

\begin{theorem} \label{T:Choquet}  (See \cite{Phelps} and 
    \cite[Section IV.6]{Takesaki}.)
    Suppose that $\cC$ is a compact convex subset of the linear 
    topological vector space $E$ and $\bv \in \cC$.  Then there is a probability 
    measure $\nu$ supported on the closure of the set of extreme points 
    $(\partial_{e} \cC)^{-}$ which represents $\bv$.  In case $\cC$ 
    is metrizable, then $\partial_{e}\cC$ is a Borel set and one can 
    arrange that $\nu$ is supported exactly on the set of extreme 
    points $\partial_{e} \cC$.
 \end{theorem}

In general, given a collection of vectors $\{\bu_{1},
\dots, \bu_{\kappa}\}$ in a linear space $E$ and given another vector
$\bv$ in $E$, we say that {\em $\bv$ is in the interior of the convex
hull of the set of vectors} ${\mathcal S} = \{\bu_{1}, \dots, \bu_{\kappa}\}$, written as
$$
   \bv \in \operatorname{conv}^{0} {\mathcal S},
$$
if $\bv$ can be written as a convex linear combination $\bv =
\sum_{i=1}^{\kappa} \lambda_{i} \bu_{i}$ of the elements of  ${\mathcal S}$
with the coefficients satisfying $\lambda_{i} > 0$,
$\sum_{i=1}^{\kappa} \lambda_{i} = 1$ uniquely determined.  We will
be particularly interested in the case when the zero vector
$\bo$ in $E$ is in the convex hull of ${\mathcal S}$.
In general there are several equivalent formulations of the condition
that
$\bo \in \operatorname{conv}^{0}{\mathcal S}$.

\begin{proposition}   \label{P:conv0}
    Given a finite subset ${\mathcal S} = \{\bu_{1}, \dots,
    \bu_{n}\}$ of a linear space $E$, suppose that the zero vector
    $\bo$ is a proper convex combination of $\bu_{1}, \dots, \bu_{n}$
    in the linear space $E$:
    $$
    \bo = \sum_{i=1}^{n} \lambda_{i} \bu_{i} \text{ with }
    \lambda_{i}>0 \text{ for all } i \text{ and } \sum_{i=1}^{n}
    \lambda_{i} = 1.
    $$
    Then the following conditions are equivalent.
    \begin{enumerate}
	\item $\bv \in \operatorname{conv}^{0} {\mathcal S}$, i.e.,
the real numbers $\lambda_{1}, \dots, \lambda_{n}$ are
uniquely determined by the conditions
\begin{equation}   \label{sys1}
  \lambda_{i} > 0 \text{ for all } i, \quad \sum_{i=1}^{n}\lambda_{i} = 1,
  \quad  \sum_{i=1}^{n} \lambda_{i} \bu_{i} = \bo.
\end{equation}

\item The linear
subspace of ${\mathbb R}^{n}$ consisting of vectors
${\mathbf c} = (c_{1}, \dots, c_{n})$ such that $c_{1}\bu_{1} + \cdots
+ c_{n} \bu_{n} = 0 \in E$ is one-dimensional (and hence is spanned
by $\lambda = (\lambda_{1}, \dots, \lambda_{n})$).

\item The only
solution ${\mathbf c} = (c_{1}, \dots, c_{n})$ of the system of
equations
\begin{equation}  \label{sys3}
 \sum_{i=1}^{n} c_{n} = 0, \quad \sum_{i=1}^{n} c_{i} {\mathbf u}_{i}
 = \bo
 \end{equation}
 is ${\mathbf c} = (0, \dots, 0)$.
\end{enumerate}
\end{proposition}

\begin{proof}  We show (not (1)) $\Rightarrow$ (not (2))
    $\Rightarrow$ (not (3)) $\Rightarrow$ (not (1)).

    \textbf{(not (1)) $\Rightarrow$ (not (2)):} Suppose that $\lambda
    = (\lambda_{1}, \dots, \lambda_{n})$ and $\lambda' =
    (\lambda'_{1}, \dots, \lambda'_{n})$ are two distinct elements of
    ${\mathbb R}^{n}$ satisfying the conditions in (1).  Set $c_{i} =
    \lambda_{i} - \lambda'_{i}$ and ${\mathbf c} = (c_{1}, \dots,
    c_{n})$.  Then ${\mathbf c}$ is a second nonzero solution of $c_{1}
    \bu_{1} + \cdots + c_{n} \bu_{n} = \bo$.  Since $\sum_{i=1}^{n}
    c_{i} = \sum_{i-1}(\lambda_{i} - \lambda'_{i}) = 0$ while
    $\lambda_{i} > 0$ for all $i$, we see that ${\mathbf c}$ is
    linearly independent of $\lambda$. Hence the set of solutions
    ${\mathbf c} = (c_{1}, \dots, c_{n})$ of $\sum_{i=1}^{n} c_{i}
    \bu_{i} = \bo$ has dimension at least 2, in contradiction to
    (2).

     \textbf{(not (2)) $\Rightarrow$ (not (3)):}  If the space of
     vectors ${\mathbf c} = (c_{1}, \dots, c_{n})$ in ${\mathbf
     R}^{n}$ given by $c_{1} \bu_{1} + \cdots + c_{n} \bu_{n}$ has
     dimension at least 2, then by the null-kernel theorem from
     Linear Algebra we can find a nonzero such vector which satisfies
     the single additional linear  constraint $c_{1} + \cdots + c_{n} =
     0$, in contradiction with (3).

      \textbf{(not (3)) $\Rightarrow$ (not (1)):}  Assume that
      ${\mathbf c} = (c_{1}, \dots, c_{n}) \in {\mathbb R}^{n}$ is a
      nonzero solution of \eqref{sys3}.  Set $\lambda'  = (\lambda_{1} +
      \epsilon c_{1}, \dots, \lambda_{n} + \epsilon c_{n})$ for some
      $\epsilon > 0$.  Then as long as $\epsilon$ is chosen
      sufficiently small, $\lambda'$ is a second solution of
      \eqref{sys1}, in contradiction with (1).
    \end{proof}

\section{An extreme-point problem for a constrained normalized ball
of matrix measures}  \label{S:extreme}

In this section we consider the following general extreme-point
problem which is central for our analysis of the matrix-valued Herglotz and Schur
classes of holomorphic functions over a finitely connected planar
domain discussed in the next section.  We suppose that we are given
a compact Hausdorff space $X$. We let $M(X)$ denote the space of
complex Borel measures on $X$ and $C_{{\mathbb R}}(X)$ denote the
space of real-valued continuous functions on $X$. For $N$ a positive
integer, $M(X)^{N \times N}$ then denotes the
space of complex $N \times N$ matrix-valued  Borel measures on $X$.
We will also have occasion to use $[M(X)^{N \times N}]_{h}$ to denote
complex Hermitian $N \times N$ matrix-valued measures and $[M(X)^{N
\times N}]_{+}$ the subset of $[M(X)^{N \times N}]_{h}$ consisting of
positive matrix measures.
We suppose that we are also given a collection $\bphi = \{ \phi_{1}, 
\dots, \phi_{m}\}$ of $m$ complex-valued continuous
functions on $X$ (i.e., $\phi_{1}, \dots, \phi_{m} \in C_{\mathbb
R}(X)$).  We then let ${\mathcal C}(X, N, \bphi)$ be the subset of $[M(X)^{N \times
N}]_{h}$ given by
\begin{align}
    {\mathcal C}(X,N, \bphi) =  &  \Big\lbrace \mu \in M(X)^{N \times N} \colon \mu(\Delta)
    \ge 0 \text{ for all Borel }, \, \mu(X) = I, \text{ and } \notag  \\
    & \quad  \mu(\phi_{r}):=
    \int_{X} \phi_{r}(x) \, {\tt d}\mu(x) = 0 \text{ for } r = 1, \dots,
    m \Big\rbrace.
    \label{cC}
\end{align}

Note that ${\mathcal C}(X,N, \bphi)$ is a convex subset of the real 
Banach space $[M(X)^{N \times N}]_{h}$
which is compact in the weak-$*$ topology on [$M(X)^{N \times N}]_{h}$
induced by its duality with respect to
the real Banach space $C_{{\mathbb R}}(X)^{N \times N}$.
In view of the Kre\u{\i}n-Milman theorem and the results discussed in
Section \ref{S:prelim}, it is then natural to pose the problem:

\begin{problem}  \label{P:extreme}  Given a data set $(X, N, \bphi)$ as above, 
    characterize the set of extreme points
    of the associated compact, convex set ${\mathcal C}(X,N, \bphi)$ given by
    \eqref{cC}.
    \end{problem}

Simple examples show that it is possible that ${\mathcal C}$ is
empty: for example, take $X$ equal to the unit interval $[0,1]$,
$N=1$, $m=1$ with $\phi_{1}(x) = 1$.  Then the condition that
$1=\mu(X) =  \int_{X} {\tt d}\mu(x)$ and that $0 = \int_{X}
\phi_{1}(x) {\tt d} \mu(x) = \int_{X} {\tt d} \mu(x)$ are
contradictory.  In the discussion to follow we will implicitly assume
that ${\mathcal C} \ne \emptyset$; in all examples arising from some
natural context, it is the case that ${\mathcal C} \ne \emptyset$.

The following result is a first step toward obtaining more definitive
solutions for various special cases of interest.

\begin{theorem}  \label{T:mu-form}
    Suppose that $\mu \in [M(X)^{N \times N}]_{h}$
    is an extreme point of ${\mathcal C(X,N, \bphi)}$ \eqref{cC}.
    Then there is a natural number $n$ with
    $1 \le n \le (m+1)N^{2}$, $n$ distinct  points $x_{1}, \dots, x_{n}$ in
    $X$, and $n$ positive semidefinite $N \times N$ matrices $W_{1},
    \dots, W_{n}$ subject to the system of linear equations
    \begin{equation}   \label{W-sys}
	 \sum_{r=1}^{n} W_{i}   = I, \quad
	 \sum_{r=1}^{n} \phi_{i}(x_{r}) W_{r}  = 0 \text{ for } i = 1,
    \dots, m
    \end{equation}
 so that $\mu$ has the form
 \begin{equation}   \label{mu-form}
   \mu = \sum_{j=1}^{n} W_{j} \delta_{x_{j}}
 \end{equation}
 where $\delta_{x_{j}}$ is the scalar-valued measure equal to the unit point-mass
 at the point $x_{j}$.
    \end{theorem}

    \begin{proof}
	It suffices to show that any extreme point $\mu =
	[\mu_{ij}]_{i,j=1, \dots, N}$ has the form
	\eqref{mu-form};  conditions \eqref{W-sys} then follow just
	by the condition that $\mu$ is an element of ${\mathcal 
	C}(X,N,\bphi)$.
	By way of contradiction, suppose that $\mu \in [M(X)^{N \times N}]_{+}$ is a positive
	matrix measure which is not of the form \eqref{mu-form}.  We
	then must show that $\mu$ is not extreme.
	
	If $\mu$ is not of the form \eqref{mu-form} with $1 \le n \le
	(m+1)N^{2}$, then  there are $\kappa$ (with $\kappa >
	(m+1)N^{2}$) disjoint Borel sets $\Delta_{1}, \dots,
	\Delta_{\kappa}$ with $\mu(\Delta_{j}) \ne 0$. 
	Define new measures $\mu_{1}, \dots,	\mu_{\kappa}$ by
	$$
	  \mu_{j}(\Delta) = \mu(\Delta \cap \Delta_{j}) \text{ for }
	  j = 1, \dots, \kappa.
	$$
 Then the collection $\{ \mu_{1}, \dots, \mu_{\kappa}\}$ is linearly
 independent in the real vector space $[M(X)^{N \times N}]_{h}$ of
 complex-Hermitian matrix-valued Borel measures on $X$.
 Now define real linear functionals
 on $[M(X)^{N\times N}]_{h}$ by
 \begin{align*}
& L_{i} \colon \mu\mapsto \mu_{ii}(X),\;1\leq i\leq N, \\
&  L_{\text{Re},ij}:\mu\mapsto Re\;\mu_{ij}(X),\;1\leq i<j\leq N, \\
 &  L_{\text{Im},ij}:\mu\mapsto Im\;\mu_{ij}(X),\;1\leq i<j\leq N, \\
&  L_{i,r}:\mu\mapsto \mu_{ii}(\phi_{r}),\;1\leq i\leq N,1\leq r\leq m,  \\
&  L_{\text{Re},ij,r}:\mu\mapsto Re\;\mu_{ij}(\phi_{r}),\;1\leq i<j\leq N,1\leq
r\leq m, \\
&  L_{\text{Im},ij,r}:\mu\mapsto Im\;\mu_{ij}(\phi_{r}),\;1\leq i<j\leq N,1\leq r\leq m.
\end{align*}
Note that in total there are
$$N+\frac{N(N-1)}{2}+\frac{N(N-1)}{2}+Nm+\frac{N(N-1)}{2}m+\frac{N(N-1)}{2}m=N^2(m+1)$$
such real linear functionals.
Note that for $1\leq i\leq j\leq N$ and $1\leq r\leq m$ we have
\begin{equation}\label{estrella}
    \mu_{ji}(X)=\mu_{ij}(X)^*
\text{ and }\mu_{ji}(\phi_{r})=\mu_{ij}(\phi_{r})^*.
\end{equation}
We now define a real linear map $L$ from $[M(X)^{N \times N}]_{h}$ to
${\mathbb R}^{(m+1)N^{2}}$ by
$$
L(\mu) = \begin{bmatrix} \operatorname{col}_{j}
\{ L_{j}(\mu) \colon 1 \le j \le N \}  \\
\operatorname{col}_{i,j}
\{L_{\text{Re}, ij}(\mu) \colon 1 \le i < j \le N\}  \\
\operatorname{col}_{i,j} \{ L_{\text{Im},
ij}(\mu) \colon 1 \le i< j \le N\}  \\
\operatorname{col}_{i,r}
\{L_{i,r}(\mu) \colon 1 \le i \le N,\, 1 \le r \le m\}   \\
\operatorname{col}_{i,j,r} \{
L_{\text{Re}, ij, r}(\mu)\colon 1 \le i < j \le N,\, 1 \le r \le m\}  \\
\operatorname{col}_{i,j,r}
\{L_{\text{Im}, ij, r}(\mu)  \colon 1 \le i < j \le N,\, 1 \le r \le
m\}
\end{bmatrix}
$$
where we use the notation $\operatorname{col}\{X_{j} \colon 1 \le j
\le N\}$ to denote the column matrix
$\operatorname{col}\{X_{j} \colon 1 \le j \le N\} = \left[
\begin{smallmatrix} X_{1} \\ \vdots \\ X_{N} \end{smallmatrix}
    \right]$.
Consider the restriction of $L$ to the $\kappa$-dimensional subspace
${\mathcal M}: = \operatorname{span} \{\mu_{1}, \dots, \mu_{\kappa}\}$.
Since $\kappa > (m+1)N^{2}$, as a consequence of the null-kernel
theorem from linear algebra we see that there exists a nonzero measure
$\nu = \sum_{\ell = 1}^{\kappa} c_{\ell} \mu_{\ell} \in \cM$
with $L(\nu) = 0$.  Consequently the matrix measure $\nu = [\nu_{ij}]_{i,j=1, \dots N}$
satisfies $\nu_{ii}(X) = 0$ for all $i =1, \dots, N$ and $\nu_{ij}(X)
= 0$ for $1 \le i < j \le N$.  From \eqref{estrella} we see that
 $\nu_{ij}(X) = 0$ for all  $1 \le i \le j
\le N$ as well and we conclude that
\begin{equation}  \label{nuX=0}
\nu(X) = 0.
\end{equation}
In a similar
way we see that in addition
\begin{equation}  \label{nufr=0}
    \nu(\phi_{r}) = 0 \text{ for } r=1, \dots, m
    \end{equation}
as well.

We next choose $\epsilon > 0$ so that $\epsilon < \operatorname{min}
\left\{ \frac{1}{c_{j}} \colon j \text{ with } c_{j} \ne 0\right\}$
where $c_{1}, \dots, c_{\kappa}$ are the coefficients in the
representation $\nu = c_{1} \mu_{1} + \cdots + c_{\kappa} \mu_{\kappa}$
for $\nu$ as an element of the space ${\mathcal M} =
\operatorname{span} \{\mu_{j} \colon j = 1, \dots,
\kappa\}$.  Then by construction
$$
  1 \pm \epsilon c_{j} \ge 0 \text{ for } j = 1, \dots, \kappa.
$$
It follows that $ (\mu \pm \epsilon\nu)(X) =
I$, $(\mu \pm \epsilon \nu)(\Delta) \ge 0$ for all Borel $\Delta$ and
$(\mu \pm \epsilon \nu)(\phi_{r}) = \mu(\phi_{r}) = 0$ for $1 \le r \le m$,
i.e., $\mu \pm \epsilon \nu \in {\mathcal C}(X,N, \bphi)$.  Since it is also the
case that $\epsilon \nu$ is not the zero element of $[M(X)^{N \times
N}]_{h}$, it follows as a consequence of Lemma \ref{L:convex} that
$\mu \notin \partial_{e}{\mathcal C}(X,N,\bphi)$, as needed to be shown.
\end{proof}

\subsection{The scalar-valued case: $N=1$}  \label{S:extreme-scalar}

We now analyze Problem \ref{P:extreme} for the scalar-valued case
($N=1$).    The
following result gives a complete characterization of
$\partial_{e}{\mathcal C}$  (${\mathcal C}$ as in \eqref{cC}) for the
scalar case ($N=1$).

\begin{theorem} \label{T:extreme-scalar}
    Suppose  that we are given a compact Hausdorff space $X$ along
    with $m$ real-valued continuous functions $\bphi = \{\phi_{1}, \dots,
    \phi_{m}\}$ and we let ${\mathcal C}(X,1, \bphi)$ be the associated compact
    convex set of scalar measures given by \eqref{cC} (with $N=1$).  Suppose that the
    positive scalar  measure $\mu$ has the form \eqref{mu-form} (tailored to
    the scalar case):
    \begin{equation}   \label{mu-form-scalar}
    \mu = \sum_{j=1}^{\kappa} w_{j} \delta_{x_{j}}
    \end{equation}
    where $x_{1}, \dots, x_{n}$ are distinct points in $X$ ($1 \le n
    \le m+1$) and $w_{j} \in {\mathbb R}$ are subject to
    \begin{equation}  \label{scalar-constraints}
    w_{j} > 0 \text{ for } 1 \le j \le n, \quad
    \sum_{j=1}^{n} w_{j} = 1, \quad \sum_{j=1}^{n}  \phi_{i}(x_{j})
    w_{j} = 0 \text{ for } i=1, \dots, m.
    \end{equation}
    Denote by $\bphi(x_{j})$ the vector $\bphi(x_{j}) = \left[
    \begin{smallmatrix} \phi_{1}(x_{j}) \\ \vdots \\ \phi_{m}(x_{j})
	\end{smallmatrix} \right]$ in ${\mathbb R}^{m}$ for $j = 1,
	\dots, n$.
   Then $\mu \in \partial_{e} {\mathcal C}(X,1,\bphi)$ if and only if
   $0 = \sum_{j=1}^{n} w_{j} \bphi(x_{j})$ is an interior point of
   the convex hull of $\{ \bphi(x_{1}), \dots, \bphi(x_{n}) \}$ in
   ${\mathbb R}^{m}$.
    \end{theorem}

    \begin{proof}  By Theorem \ref{T:mu-form} tailored to the
	scalar-valued case, we know that any $\mu \in
	\partial_{e}{\mathcal C}$ has the form \eqref{mu-form-scalar}
	with base points $x_{1}, \dots, x_{n}$ and weights $w_{1},
	\dots, w_{n}$ subject to \eqref{scalar-constraints}; the
	question is: which such $\mu$'s are actually extreme points?
	
	Note that conditions \eqref{scalar-constraints} can be
	interpreted as exhibiting $0 \in {\mathcal R}^{m}$ as lying in
	the convex hull of $\{ \bphi(x_{1}), \dots, \bphi(x_{n}) \}$.  It remains
	to show that the measure $\mu = \sum_{j=1}^{n} w_{j}
	\delta_{x_{j}}$ is an extreme point of ${\mathcal C}(X,1, \bphi)$ if and
	only if in fact $0$ is in the interior of the convex hull of
	$\{\bphi(x_{1}), \dots, \bphi(x_{n}) \}$.  
	
	Let us suppose the $0 = \sum_{j=1}^{n} w_{j} \bphi(x_{j})$ is
	{\em not} in the interior of the convex hull.  By statement
	(3) in Proposition \ref{P:conv0}, this is the same as the
	existence of real numbers $c_{1}, \dots, c_{n}$ not all zero
	with
	$$
	\sum_{j=1}^{n} c_{j} = 0, \quad \sum_{j=1}^{n} c_{j}
	\bphi(x_{j}) = 0.
	$$
	Define a measure $\nu =  \epsilon \sum_{j=1}^{n}
	c_{j}\delta_{x_{j}}$.  Then $\nu \ne 0$, $\nu(X) = 0$ and $\nu(\phi_{j}) =
	\int_{X} \phi_{j} {\tt d}\nu = \sum_{j=1}^{n} c_{j}
	\phi(x_{j}) = 0$. If we choose $\epsilon > 0$ sufficiently
	small, then $\mu  \pm \nu \in {\mathcal C}(1,X, \bphi)$.  We conclude by
	Lemma \ref{L:convex} that $\mu$ is not extremal in ${\mathcal
	C}(X,1,\bphi)$.
	
	Suppose next that $0 = \sum_{j=1}^{n} w_{j} \bphi(x_{j})$ is
	an interior point of the convex hull of $\{\bphi(x_{1}), \dots,
	\bphi(x_{n}) \}$ in ${\mathbb R}^{m}$. We wish to show that then $\mu
	\in \partial_{e}{\mathcal C}(X,1, \bphi)$.  We therefore suppose that
	$\mu = t_{1} \mu_{1} +  t_{2} \mu_{2}$ with $\mu_{k} \in
	{\mathcal C}(X,1,\bphi)$ and $t_{k}> 0$ for each $k=1,2$ and $t_{1} + t_{2} = 1$.  Since
	$\mu_{k}$ is a positive measure for each $k$, we read off
	from \eqref{mu-form-scalar} that $\operatorname{supp} \mu_{k}
	\subset \{x_{1}, \dots, x_{n}\}$, so each $\mu_{k}$ has the
	form $\mu_{k} = \sum_{j=1}^{n} w_{j}^{(k)} \delta_{x_{j}}$
	for some weights $w_{j}^{(k)} \ge 0$ with $\sum_{j=1}^{n}
	w_{j}^{(k)} = 1$.  From the fact that $\mu_{k} \in {\mathcal
	C}(X,1,\bphi)$ we also have that $\mu_{k}(\phi_{i}) = \sum_{j=1}^{n}
	w^{(k)} \phi_{i}(x_{j}) = 0$ for each $i = 1, \dots, m$, or, in vectorial form,
	$ \sum_{j=1}^{n} w^{(k)}_{j} \bphi(x_{j}) = 0 \in {\mathbb
	R}^{m}$.  By the assumption that $0 = \sum_{j=1}^{n} w_{j}
	\bphi(x_{j})$ is in interior point for the convex hull of
	$\{\bphi(x_{1}), \dots, \bphi(x_{n})\}$ in ${\mathbb R}^{m}$,
	statement (1) in Proposition \ref{P:conv0} gives us
        that $w_{j}^{(k)} = w_{j}$ for $j=1, \dots, n$ for
	each $k = 1,2$, i.e., $\mu_{k} = \mu$.  We conclude that
	$\mu$ is indeed an extreme point of ${\mathcal C}(X,1,\bphi)$ as wanted.
	\end{proof}

\begin{corollary}  \label{C:extreme-scalar}  Suppose that we are
    given a data set
 $$  X = \text{ a compact Hausdorff space}, \quad \bphi = \{\phi_{1}, \dots,
 \phi_{m}\} \subset C_{{\mathbb R}}(X)
 $$
 and we let $\cC(X,1,\bphi)$ be as in \eqref{cC} (with $N=1$).  Then
 ${\mathcal C}(X,1,\bphi) \ne \emptyset$ if and only if, for some natural number
 $n$ with $1 \le n \le m+1$, there exists a collection of
 $n$ distinct  points $x_{1}, \dots, x_{n}$ in $X$ such that $0 =
 \sum_{j=1}^{n} w_{j} \bphi(x_{j})$ is an interior point of
 the convex hull of the set $\{ \bphi(x_{1}), \dots, \bphi(x_{n})
 \}$ in ${\mathbb R}^{m}$,
 where $\bphi(x_{j}) = \left[ \begin{smallmatrix}
 \phi_{1}(x_{j}) \\ \vdots \\ \phi_{m}(x_{j}) \end{smallmatrix}
 \right]$.
 \end{corollary}

 \begin{proof} By the Kre\u{\i}n-Milman theorem, ${\mathcal 
     C}(X,1,\bphi)$ has
     extreme points if and only if ${\mathcal C}(X,1,\bphi)$ is not empty.  The
     conclusion is now immediate from Theorem \ref{T:extreme-scalar}.
 \end{proof}

 \begin{remark}   \label{R:m=0scalar}
    {\em  One can interpret Theorem \ref{T:extreme-scalar} even for the
     case $m=0$.  In this case $1 \le n \le m+1 = 1$ forces $n=1$.
     The constraints \eqref{scalar-constraints} force $\mu$ to have
     the form $\mu = \delta_{x}$ for some $x \in X$.  As there are no
     $\phi$'s, the condition that $0$ be an interior point of the 
     convex hull of $\{ \bphi(x_{1}), \dots, \bphi(x_{n})\}$ can be 
     interpreted to hold vacuously
     This recovers the correct result that the set of
     extreme points of the normalized matrix ball $\{\mu \in  M(X)_{+}
     \colon \mu(X) = 1\}$ consists of the unit point masses $\{
     \delta_{x} \colon x \in X\}$.}
   \end{remark}

   \begin{remark}  \label{R:Pickering}
       {\em  Another special case of Theorem
       \ref{T:extreme-scalar}  of interest is the case where
       $X = {\mathbb T}$ is the unit circle in the complex plane,
       $m=2$ with $\phi_{1}(z) = \R z$ and $\phi_{2}(z) = \I z$.
      In this case one can give explicit geometric characterization
      of $\partial_{e}{\mathcal C}$.  Indeed, pairs of points with
      $\bo
      \in {\mathbb R}^{2} \cong {\mathbb C}$ correspond to antipodal
      points on the unit circle, and triples of points $x_{1}, x_{2},
      x_{3}$ on the unit circle with $\bo \in
      \operatorname{conv}^{0}\{x_{1}, x_{2}, x_{3}\}$ amount to
      non-collinear points on the unit circle having $\bo \in {\mathbb
      C}$ in the interior of the simplex spanned by $x_{1}, x_{2},
      x_{3}$. This analysis has been worked out by Dritschel and Pickering in
      \cite{DP}.  Motivation for this example comes from the search
      for a collection of test functions for the Schur class
      associated with the  constrained $H^{\infty}$ class 
      $H^{\infty}_{1}  = \{ f \in
      H^{\infty} \colon f'(0) = 0\}$ (see \cite{DPRS} and Remark 
      \ref{R:testfunc} below).   In this context there is an
      additional equivalence relation imposed on
      $\partial_{e}{\mathcal C}({\mathbb T}, 1,\bphi)$ and the set of 
      equivalence classes of ${\mathcal 
      C}({\mathbb T}, 1,\bphi)$ can be identified
      topologically with the unit sphere.

    } \end{remark}

    \subsection{Return to the general matrix-valued case}
    \label{S:return}

    We now indicate how one can analyze the general case of Problem
    \ref{P:extreme} by using the language of noncommutative convexity
    (see \cite{LP, HMP, FM, Gregg}).

    Rather than delve into the general setting of $C^{*}$-convex
    combinations of elements of a $C^{*}$-algebra or of the
    generalized state space of a $C^{*}$-algebra and associated
    $C^{*}$-convex subsets and $C^{*}$-extreme points, we discuss
    only the concrete special case which we need for our application
    (but see Remark \ref{R:noncomcon} below).
    We fix a positive integer $N$ and consider a collection of $n$
    vectors in the space ${\mathcal X}: = ([{\mathbb C}^{N \times
    N})]_{h})^{m \times 1}$, i.e., column vectors of length $m$, each entry
    of which is an $N \times N$ complex Hermitian matrix.  Given a
    collection of $n$ elements $\Phi^{(1)}, \dots, \Phi^{(n)}$ in
    ${\mathcal X}$, we say that $\Phi \in {\mathcal X}$ is a {\em
    $C^{*}$-convex combination} of $\Phi^{(1)}, \dots, \Phi^{(n)}$ if
    there are $n$ matrices $A_{1}, \dots, A_{n}$ of size $N \times N$
    with $\sum_{j=1}^{n} A_{j}^{*} A_{j} = I$ so that
    \begin{equation}   \label{C*convcomb1}
    \Phi = \sum_{j=1}^{n} A_{j}^{*} \Phi^{(j)} A_{j}
    \end{equation}
    where we set
    \begin{equation}  \label{C*convcomb2}
    A_{j}^{*} \Phi^{(j)} A_{j} = \begin{bmatrix} A_{j}^{*}
    \Phi^{(j)}_{1} A_{j} \\ \vdots \\ A_{j}^{*} \Phi^{(j)}_{m}A_{j} \end{bmatrix}
    \text{ if } \Phi^{(j)} = \begin{bmatrix} \Phi^{(j)}_{1} \\ \vdots \\
    \Phi^{(j)}_{m} \end{bmatrix} \in ([{\mathbb C}^{N \times
    N}]_{h})^{m \times 1}.
    \end{equation}
    For our application, we only deal with the special case where
    $\Phi^{(j)}_{i}$ is a
    scalar multiple of the identity:  $\Phi^{(j)}_{i} =
    \phi^{(j)}_{i} I_{N}$ where $\phi^{(j)}_{i}$ is a real number; we
    denote the subspace of all such elements of ${\mathcal X}$ by
    ${\mathcal X}_{s}$.  For $\Phi^{(1)}, \dots, \Phi^{(n)} \in
    {\mathcal X}_{s}$,  the $C^{*}$-convex combination
    \eqref{C*convcomb1} and \eqref{C*convcomb2} simplifies to
    \begin{equation}  \label{C*conv1}
    \Phi = \sum_{j=1}^{n}  W_{j} \Phi^{(j)}
    \end{equation}
    where we set $W_{j} = A_{j}^{*} A_{j}$, so $\{W_{j} \colon j=1,
    \dots n\}$ is any collection of $N \times N$ matrices satisfying
    \begin{equation}   \label{C*conv2}
	W_{j} \ge 0 \text{ for } j = 1, \dots, n, \quad
	\sum_{j=1}^{n} W_{j} = I_{N}
   \end{equation}
   and the meaning of the $j$-th term in \eqref{C*conv1} is
   \begin{equation}  \label{C*conv3}
       W_{j} \Phi^{(j)} = \begin{bmatrix} \phi^{(j)}_{1} W_{j} \\
       \vdots \\ \phi^{(j)}_{m} W_{j} \end{bmatrix} \text{ if }
       \Phi^{(j)} = \begin{bmatrix}  \phi^{(j)}_{1} I_{N} \\ \vdots
       \\ \phi^{(j)}_{m} I_{N} \end{bmatrix}.
   \end{equation}
    We will furthermore only be interested in the case where the
    $C^{*}$-convex combination of such $\Phi^{(1)}, \dots,
    \Phi^{(n)}$ in ${\mathcal X}_{s}$ is the zero element $\bo$ in ${\mathcal X}$:
    $$  \bo = \begin{bmatrix} 0 \\ \vdots \\ 0 \end{bmatrix} \in
    ([{\mathbb C}^{N \times N}]_{h})^{m \times 1}.
    $$

   In analogy with the notion of the interior point of the convex hull of
    a collection of vectors $\bu_{1}, \dots, \bu_{n}$ for the
    classical case presented in Section \ref{S:prelim}, we propose
    the following definition of the notion that $\bo$ is an interior
    point of the  $C^{*}$-convex hull of a collection of vectors in
    ${\mathcal X}_{s}$. The statement of the result requires some
    additional terminology, all of which we collect in the following
    definition.

    \begin{definition}   \label{D:inthull}
{\em 	Given
    an operator $T$ on a Hilbert space ${\mathcal H}$ (e.g.,
    ${\mathcal H} = {\mathbb C}^{N}$ and $T$ presented as a matrix in
    ${\mathbb C}^{N \times N}$) together with a closed
    subspace ${\mathcal M}$ of ${\mathcal H}$, we  say that
     $T$ \textbf{lives on} ${\mathcal M}$ if $T=TP_{{\mathcal M}} =
    P_{{\mathcal M}} T$ (where $P_{{\mathcal M}}$ is the orthogonal
    projection from ${\mathcal H}$ to ${\mathcal M}$).

    Given a family of closed subspaces ${\mathcal M}_{1}, \dots, {\mathcal
    M}_{n}\}$ of ${\mathcal H}$, we say that the family
    $\{{\mathcal M}_{1}, \dots, {\mathcal M}_{n}\}$ is \textbf{weakly
    independent} if, {\em whenever $T_{1}, \dots, T_{n}$ are linear
    operators on ${\mathcal H}$ with
    \begin{equation}  \label{weakindhy}
    T_{j} \text{ lives on } {\mathcal M}_{j} \text{ for each } j
    \text{ and } \sum_{j=1}^{n} T_{j} = 0,
    \end{equation}
    it follows that $T_{j} = 0$ for each $j=1, \dots, n$}.

    Suppose that in addition we are given a collection $\bphi = \{
    \phi^{(1)}, \dots, \phi^{(n)}\}$ of $n$ vectors in ${\mathbb
    R}^{m}$ (so $\phi^{(j)} = \left[ \begin{smallmatrix}
    \phi^{(j)}_{1} \\ \vdots \\ \phi^{(j)}_{m} \end{smallmatrix}
    \right]$ with real numbers $\phi^{(j)}_{1}, \dots,
    \phi^{(j)}_{m}$ say).
     Then we say that the family of closed
    subspaces $\{{\mathcal M}_{1}, \dots, {\mathcal M}_{n}\}$ is
     $\bphi$-\textbf{constrained weakly independent} if, {\em whenever
    $T_{1}, \dots, T_{n}$ are linear operators on ${\mathcal H}$ with
    \begin{equation}   \label{phiweakindhy}
	T_{j} \text{ lives on } {\mathcal M}_{j} \text{ for each } j,
    \, \,
     \sum_{j=1}^{n} T_{j} = 0, \, \,  \text{ and } \sum_{j=1}^{n}
     \phi^{(j)}_{i} T_{j} = 0 \text{ for each } i=1, \dots, m,
     \end{equation}
     it follows that $T_{j} = 0$ for each $j=1, \dots, n$}.

     Finally, suppose that we are given $n$-vectors $\bphi = \left\{ \left[
     \begin{smallmatrix} \phi^{(j)}_{1} \\ \vdots \\ \phi^{(j)}_{m}
	 \end{smallmatrix} \right] \colon j = 1, \dots, n\right\}$ in
	 ${\mathbb R}^{m}$ with associated set of $n$ vectors
	 ${\boldsymbol \Phi} = \left\{  \left[
     \begin{smallmatrix} \phi^{(j)}_{1} I_{N} \\ \vdots \\
	 \phi^{(j)}_{m} I_{N}
	 \end{smallmatrix} \right] \colon j = 1, \dots, n\right\}$
	 in  ${\mathcal X}_{s}$, and suppose that
     $0$ is in the $C^{*}$-convex
     hull of the $\boldsymbol \Phi$: there are matrices $W_{1}, \dots,
     W_{n}$ satisfying conditions \eqref{C*conv2} so that
     $$
       0 = \sum_{j=1}^{n} W_{j} \Phi^{(j)}
     $$
       with $W_{j} \Phi^{(j)}$ defined as in \eqref{C*conv3}.  Then
       we say that $0$ \textbf{is an interior point of the $C^{*}$-convex
       hull} of $\{ \Phi^{(j)} \colon j=1, \dots, n\}$ if {\em the family
       of subspaces $\{\operatorname{Ran} W_{1}$, $\dots,$
      $ \operatorname{Ran} W_{n}\}$ is $\bphi$-constrained weakly
       independent.}
     	}\end{definition}
	
	An easy observation is that for the case $N=1$, the notion of 
	$0$ being an interior point of the $C^{*}$-convex hull of 
	${\boldsymbol \Phi} = \{ \phi^{(1)}, \dots, \phi^{(n)}\} 
	\subset {\mathbb R}^{m}$ 
	coincides with $0$ being an interior point of the convex hull 
	of $\{ \phi^{(1)}, \dots, \phi^{(n)} \}$ as characterized in 
       Proposition \ref{P:conv0}.  Indeed, 
	supposes that $0 = \sum_{k=1}^{n} w_{k} \phi^{(k)}$ for 
	positive numbers $w_{1}, \dots, w_{n}$ summing to $1$, and 
	$t_{1}, \dots, t_{n}$ is a collection of real numbers with
	$$
	\sum_{k=1}^{n} t_{k} = 0, \quad \sum_{k=1}^{n} t_{k} 
	\phi^{(k)} = 0 \in {\mathbb R}^{m}.
	$$
	Since $\operatorname{Ran} w_{k}$ is the whole space ${\mathbb 
	C}$ (when $w_{k}$ is  considered as an operator on ${\mathbb 
	C}$), it is automatically the case that $t_{k}$  ``lives in''
	$\operatorname{Ran} w_{k}$.  Thus the condition for $0$ being 
	an interior point of the $C^{*}$-convex hull of 
	$\{\phi^{(1)}, \dots, \phi^{(n)}\}$ reduces to condition (3) 
	in Proposition \ref{P:conv0} (with $\phi^{(j)}$ in place of 
	$\bu_{j}$), i.e., to $0$ being an interior point of the 
	classical convex hull of $\{\phi^{(1)}, \dots, \phi^{(n)}\}$.
	
	\smallskip 
	
	We are now ready to state the following general result concerning
	Problem \ref{P:extreme}.
	
	\begin{theorem} \label{T:extreme-matrix} Let the convex set 
	    of measures 
	    ${\mathcal C} = {\mathcal C}(X, N, \bphi \}$ be given as in \eqref{cC}. 
	    Then a measure
	    $\mu$ in ${\mathcal C}$ is extremal ($\mu \in
	    \partial_{e}{\mathcal C}$) if and only if there is a
	    natural number $n$ with $1 \le n \le (m+1)N^{2}$ and $n$
	    distinct points $\bx = (x_{1}, \dots, x_{n})$ in $X$ together
	    with $N \times N$ matrix weights $W_{1}, \dots, W_{n}$
	    satisfying the conditions \eqref{W-sys} so that
	   $\mu$ has a
	    representation as in Theorem \ref{T:mu-form} (see
	    \eqref{mu-form})
	    \begin{equation}  \label{mu-form''}
	    \mu = \sum_{j=1}^{n} W_{j} \delta_{x_{j}}
	    \end{equation}
	    where, in addition, the family of subspaces $\{
	    \operatorname{Ran} W_{1}, \dots, \operatorname{Ran} W_{n}
	    \}$ is $\bphi(\bx)$-con\-strained weakly independent, where we 
	    set
	    $$
	    \bphi(\bx) = \{ \phi(x_{1}), \dots, \phi(x_{n}) \}.
	    $$
	
	    \end{theorem}
	
	    \begin{proof}
		Suppose first that $\mu$ has the form
		\eqref{mu-form''} with $\{\operatorname{Ran}
		W_{1}, \dots, \operatorname{Ran} W_{n}\}$  a
		$\bphi$-constrained weakly independent family of
		subspaces, and also suppose that $\nu$ is a complex
		Hermitian $N \times N$-matrix measure on $X$ such that
		\begin{equation}  \label{nucond}
\nu(X) = 0, \quad \nu(\phi_{i}) = \int_{X} \phi_{i} {\tt d}\nu = 0
\text{ for } i=1, \dots, m, \quad \mu \pm \nu \ge 0.
\end{equation}
From the last of conditions \eqref{nucond} we see that
$\operatorname{supp} \nu \subset \{ x_{1}, \dots, x_{m} \}$ and hence
there are complex Hermitian matrices $T_{1}, \dots, T_{n}$ so that
$$
   \nu = \sum_{j=1}^{n} T_{j} \delta_{x_{j}}.
 $$
 By evaluating $\mu \pm \nu$ on the singleton Borel set  $\{x_{j}\}$,
 we see that $W_{j} \pm T_{j} \ge 0$.  This enables us to conclude
 that  $T_{j}$ lives on $\operatorname{Ran} W_{j}$ for each $j$.
 From the first two conditions in \eqref{nucond} we deduce that
 $$
 \sum_{j=1}^{n} T_{j} = 0, \quad \sum_{j=1}^{n} \phi_{i}(x_{j}) T_{j}
 = 0 \text{ for } i=1, \dots, m.
 $$
 From the hypothesis that $\{ \operatorname{Ran} W_{j} \colon j=1,
 \dots, n\}$ is $\bphi(\bx)$-constrained weakly independent, we conclude
 that $T_{j} = 0$ for each $j$, and hence $\nu = 0$.  From the
 criterion in Lemma \ref{L:convex}, we now conclude that $\mu$ is
 extremal as wanted.

 Conversely, suppose that $\mu \in \partial_{e} {\mathcal C}$ and
 suppose that $\{ T_{1}, \dots, T_{n}\}$ is a collection of operators
 satisfying the conditions \eqref{phiweakindhy} (with
 $\operatorname{Ran} W_{j}$ in place of ${\mathcal M}_{j}$).  Note
 that $\{ \R T_{1}, \dots, \R T_{n}\}$ and $\{ \I T_{1}, \dots,
 \I T_{n}\}$ satisfying the same hypotheses and in order to show
 that $T_{j} = 0$ it suffices to show that $\R T_{j} = 0$ and $\I
 T_{j} = 0$.  Thus without loss of generality we may assume that
 $T_{j}$ is complex Hermitian.  Define a measure $\nu$ by $\nu =
 \epsilon \sum_{j=1}^{n} T_{j} \delta_{x_{j}}$ where $\epsilon > 0$.
 One can check that then $\nu$ meets all the conditions
 \eqref{nucond} as long as $\epsilon > 0$ is chosen sufficiently
 small.  If $\mu$ is extremal, then Lemma \ref{L:convex} forces  $\nu
 = 0$.  As we were careful to arrange that $\epsilon \ne 0$, it
 follows that $T_{j} = 0$ for each $j=1, \dots, n$.  It now follows
 that indeed $\{ \operatorname{Ran} W_{j} \colon j=1, \dots, n\}$ is
 $\bphi(\bx)$-constrained weakly independent as was to be shown.
\end{proof}

It is of interest to specialize Theorem \ref{T:extreme-matrix} to the
case $m=0$; in this way we recover a result of Arveson (see
\cite[Theorem 1.4.10]{Arv}).
	
\begin{corollary}  \label{C:Arveson}
    Let ${\mathcal C} = \cC(X, N, \emptyset)$ be the cone of positive $N
    \times N$-matrix measures $\mu$ on a compact Hausdorff space $X$
    normalized to have $\mu(X) = I_{N}$.  Then $\mu$ is extremal in
    ${\mathcal C}$ if and only if, for some natural number $n$ with
    $1 \le n \le N^{2}$, there are $n$ distinct points $x_{1}, \dots,
    x_{n}$ and $N \times N$ matrix weight $W_{1}, \dots, W_{n}$
    satisfying
    \begin{enumerate}
	\item[(i)] $W_{j} \ge 0$ for each $j = 1, \dots, n$ and
	$\sum_{j=1}^{n} W_{j} = I_{N}$, and
	\item[(ii)] the family of subspaces $\{\operatorname{Ran}
	W_{1}, \dots, \operatorname{Ran} W_{n} \}$ is weakly
	independent
  \end{enumerate}
  so that $\mu$ is given by
  \begin{equation}   \label{mu-form-Arv}
    \mu = \sum_{k=1}^{n} W_{k} \delta_{x_{k}}.
  \end{equation}
 \end{corollary}
 
 \begin{proof}  Simply observe that this is just the $m=0$ case
     of Theorem \ref{T:extreme-matrix}.  We note that our proof
     (i.e., the proof of Theorem \ref{T:extreme-matrix} specialized
     to the $m=0$ case) is elementary and direct while the proof in
     \cite{Arv} has a more sophisticated flavor bringing in ideas
     from $C^{*}$-representation and dilation theory.
 \end{proof}
 
  \begin{remark}  \label{R:Arveson}{\em In the paper of Arveson 
     \cite{Arv}, it is not noted explicitly that the number of terms 
     $n$ in the decomposition \eqref{mu-form-Arv} for an extremal 
     measure of ${\mathcal C}(X, N, \emptyset)$ can be at most 
     $N^{2}$ for the finite-dimensional case (${\mathcal H} = 
     {\mathbb C}^{N}$).  However it is observed there (see \cite[page 
     165]{Arv}) that, in the 
     finite-dimensional case,  weak 
     independence of a family of subspaces $\{ {\mathcal M}_{1}, 
     \dots, {\mathcal M}_{n}\} \subset {\mathbb C}^{N}$ is equivalent 
     to classical linear independence for the family of subspaces 
     $\{ {\mathcal N}_{1}, \dots, {\mathcal 
     N}_{n} \} \subset {\mathcal C}^{N} \otimes {\mathcal 
     C}^{N}$, where we have set ${\mathcal N}_{j} = 
     \operatorname{span} \{ \xi \otimes \eta \colon \xi, \eta \in 
     {\mathcal M}_{j}\}$.  Since $\operatorname{dim} ({\mathbb C}^{N} 
     \otimes {\mathbb C}^{N})$ is $N^{2}$, we have the bound $N^{2}$ 
     on the number of subspaces in a weakly independent family of 
     subspaces $\{ {\mathcal M}_{1}, \dots, {\mathcal M}_{n}\}$ 
     contained in ${\mathbb C}^{N}$.  
     
     There is also an example given in \cite{Arv} of a weakly 
     independent family of subspaces which is not linearly 
     independent in the classical sense, e.g., ${\mathcal M}_{1} = 
     \operatorname{span}\{\xi \}$,
     ${\mathcal M}_{2} = 
     \operatorname{span}\{\eta \}$, ${\mathcal M}_{3}, = 
     \operatorname{span}\{\xi + \eta \}$ where $\xi$ and $\eta$ are 
     linearly independent vectors.  This example can be enhanced as 
     follows (see \cite[pages 32--35]{Mo}).  One can choose three 
     vectors $\xi_{1}, \xi_{2}, \xi_{3}$ in ${\mathbb C}^{2}$ so that 
     the family of subspaces $\cM_{j} = \operatorname{span} 
     \{\xi_{j}\}$ ($j = 1, 2, 3)$ is weakly independent and in 
     addition the associated matrix weights $W_{j} = \xi_{j} 
     \xi_{j}^{*}$ ($j = 1, 2, 3$) satisfy the normalization $W_{1} + 
     W_{2} + W_{3} = I_{2}$.  We conclude that the measure $\mu = W_{1} 
     \delta_{x_{1}} + W_{2} \delta_{x_{2}} + W_{3 }\delta_{x_{3}}$ 
     (where $x_{1}, x_{2}, x_{3}$ are any three distinct points in 
     $X$) is extremal in $\cC(X, 2, \emptyset)$ while not being  a
     spectral measure, i.e., $\mu$ is not of the form $\mu =  P_{1} \delta_{x_{1}} 
     + P_{2} \delta_{x_{2}}$ with $P_{1}, P_{2}$ orthogonal 
     projections with pairwise orthogonal ranges in ${\mathbb C}^{2}$
     (compare with Theorem \ref{T:mu-form2} below).
     } \end{remark}

	    \begin{remark}  \label{R:noncomcon} {\em The notion of
		$C^{*}$-convex combination and associated notions of
		$C^{*}$-convex set and $C^{*}$-extremal point are
		defined more broadly in the literature than what we
		have indicated so far here.  The setting of \cite{FM,
		Gregg} is the generalized state space $S_{{\mathcal
		H}}(A)$ of unit-preserving completely positive maps
		from the $C^{*}$-algebra $A$ into the $C^{*}$-algebra
		${\mathcal L}({\mathcal H})$ of bounded linear
		operators on the Hilbert space ${\mathcal H}$.  A
		$C^{*}$-convex combination of $n$ such maps
		$\phi_{1}, \dots, \phi_{n}$ is defined to be a $\phi$
		given by
		$$
		\phi(a) = \sum_{j=1}^{n} t_{j}^{*} \phi_{j}(a) t_{j}
		$$
		where $t_{j} \in {\mathcal L}({\mathcal H})$ satisfy
		$\sum_{j=1}^{n} t_{j}^{*}t_{j} = I_{{\mathcal H}}$.
		The main interest in \cite{FM, Gregg} (as well as in
		other papers) is the structure of $C^{*}$-extreme
		points of $S_{{\mathcal H}}(A)$.  More broadly, one
		could consider real maps $\phi \colon A \to {\mathcal
		L}({\mathcal H})$, i.e., maps which preserve
		selfadjoint elements, and in particular, examine when
		the zero map $0$ is a $C^{*}$-convex combination of
		$n$ given such maps $\phi_{1}, \dots, \phi_{n}$.
		This becomes exactly the setting introduced in
		Section \ref{S:return} if we take $A$ to be the
		$C^{*}$-algebra of continuous functions on the
		finite-point set $\{1, \dots, m\}$, i.e., $A = C(\{1,
		\dots, m\}) \cong {\mathbb C}^{m}$ (so selfadjoint
		elements are identified with ${\mathbb R}^{m}$),
		${\mathcal H} = {\mathbb C}^{N}$, and identify
		${\mathcal X} = ([{\mathbb C}^{N \times N}]_{h})^{m
		\times 1}$ with maps from $C(\{1, \dots, m\})$ into
		${\mathcal L}({\mathbb C}^{N}) \cong {\mathbb C}^{N
		\times N}$:
$$
\Phi = \begin{bmatrix} \Phi_{1} \\ \vdots \\ \Phi_{m} \end{bmatrix}
\in {\mathcal X} \mapsto \phi_{\Phi} \colon f \in C(\{1, \dots, m\})
\mapsto f(1) \Phi_{1} + \cdots + f(m) \Phi_{m}.
$$
We have not seen the notion of {\em interior point of the
$C^{*}$-convex hull} elsewhere in the literature.  Note that we
define this notion here only for the special case where $\Phi_{1},
\dots, \Phi_{n}$ are in ${\mathcal X}_{s}$; we do not hazard a guess
here as to what the appropriate notion should be for the more
noncommutative situation where $\Phi_{1}, \dots, \Phi_{n} \in
{\mathcal X}$, or for the still more general situation where
$\Phi_{1}, \dots, \Phi_{n}$ are real elements of ${\mathcal L}(A,
{\mathcal L}({\mathcal H}))$.

As observed in \cite{Gregg}, given a compact Hausdorff space $X$, the 
generalized state space $S_{{\mathcal H}}(C(X))$ of the commutative 
$C^{*}$-algebra $C(X)$ can be identified with positive ${\mathcal 
L}({\mathcal H})$-valued measures $\mu$ on $X$ having total mass 
$\mu(X)$ equal to $I_{{\mathcal H}}$.  Thus, when ${\mathcal H} = 
{\mathbb C}^{N}$,  $S_{{\mathbb C}^{N}}(C(X))$ is exactly the convex 
set ${\mathcal C}(X, N, \emptyset)$ whose classical extreme points 
are described in Corollary \ref{C:Arveson}.  One of the central goals in \cite{FM, 
Gregg} is to describe the $C^{*}$-extreme points of $S_{{\mathcal 
H}}(C(X))$.  It is interesting that the problem of describing the 
classical extreme points  of the linearly-constrained generalized 
state space ${\mathcal C}(X, N, \bphi)$, a 
problem formulated completely in the confines of classical convexity 
theory, has a solution (see Theorem \ref{T:extreme-matrix})  which 
draws on ideas from noncommutative convexity theory.

We note that other papers (e.g.~\cite{LP, HMP, FM-PAMS}) study
$C^{*}$-convex sets (and associated extremal-point theory) in
${\mathcal L}({\mathcal H})$ or, more generally, in a general
$C^{*}$-algebra $A$.  From our point of view this amounts to the
special case $m=1$.
}\end{remark}

It can be argued that the characterization of $\partial_{e}{\mathcal
C}(N,X,\bphi)$ in Theorem \ref{T:extreme-matrix} is not particularly explicit
and is a little difficult to work with.  To compensate for this we
now give a couple of illustrative more concrete classes of examples.

    \begin{theorem}  \label{T:mu-form1}  Suppose that ${\mathcal
	C} = {\mathcal C}(X, N, \bphi)$ is as in
	\eqref{cC}.  Suppose that $\mu \in [M(X)^{N \times N}]_{h}$
	has the form
	\begin{equation}   \label{mu-form'}
	\mu = \sum_{k=1}^{n} \mu_{k} L_{k}
	\end{equation}
	where
	\begin{enumerate}
	    \item[(i)]
	each $\mu_{k}$ is a scalar positive measure which is
	an extreme point for the associated convex compact subset ${\mathcal C}^{1}: =
    {\mathcal C}(X, 1, \bphi)$ of positive scalar measures where in 
    addition the support sets ${\mathcal S}_{k}: = \{
    \operatorname{supp} \mu_{k}
    \colon k = 1, \dots, n\}$ are disjoint (${\mathcal S}_{k} \cap
    {\mathcal S}_{k'} = \emptyset$ for $k \ne k'$),
    \item [(ii)]
     the matrix weights $L_{k}$ satisfy the conditions
     $$
  L_{k} \ge 0 \text{ for each } k, \quad  \sum_{k=1}^{n} L_{k} = I_{N}.
  $$
  and
  \item[(iii)] the family of subspaces
  $\{\operatorname{Ran} L_{k} \colon k=1, \dots, n\}$ is weakly
  independent (as defined in Definition \ref{D:inthull}).
  \end{enumerate}
  Then  $\mu \in \partial_{e}{\mathcal C}$.
 \end{theorem}

 \begin{proof} We argue that if $\mu$ is as in the statement of the
     Theorem, then it meets the conditions of Theorem
     \ref{T:extreme-matrix} and therefore is extremal in ${\mathcal
     C}$.  Toward this end, we let $\{x_{k,1}, \dots, x_{k,n_{k}}\}$
     be the support of the measure $\mu_{k}$.  By hypothesis, these
     points are all distinct.  We then write $\mu_{k}$ as
     $$
     \mu_{k} = \sum_{j=1}^{n_{k}} w^{(k)}_{j} \delta_{x_{k,j}}
     $$
     for scalar weights $w^{(k)}_{j} > 0$ and then rewrite $\mu$ as
     \begin{equation}   \label{mu-form1}
     \mu = \sum_{k=1}^{n} \sum_{j=1}^{n_{k}} w^{(k)}_{j} L_{k}
     \delta_{x_{k,j}} = \sum_{k,j \colon 1 \le k \le n; 1 \le j \le
     n_{k}} W_{k,j} \delta_{x_{k,j}}
     \end{equation}
     where we have set  $W_{k,j} = w^{(k)}_{j} L_{k}$.
     Then \eqref{mu-form1} represents $\mu$ in the form
     \eqref{mu-form''}, but with index set $\{(i,j) \colon 1 \le i
     \le n, \, 1 \le j \le n_{k}\}$ rather than $\{j \colon 1 \le j
     \le n\}$.
     it remains only to show that the collection of subspaces
     $\{\operatorname{Ran} W_{k,j} \colon 1 \le k \le n, \, 1 \le j
     \le n_{k} \}$
     is $\bphi$-constrained weakly independent.

     We therefore suppose that we are given a collections of
     operators $T_{k,j}$ on ${\mathbb C}^{N}$ satisfying
     \begin{align*}
	 & \sum_{k,j \colon 1 \le k \le n; 1 \le j \le n_{k}} T_{k,j}
	 = 0, \\
	 & \sum_{k,j\colon 1 \le k \le n; 1 \le j \le n_{k}}
	 \phi_{i}(x_{k,j}) T_{k,j} = 0 \text{ for } i = 1, \dots, m
     \end{align*}
     with the goal to show that each $T_{k,j} = 0$.
     From the hypothesis that $\{ \operatorname{Ran} L_{k} \colon k =
     1, \dots, n\}$ is weakly independent and the observation that
     both $\sum_{j=1}^{n_{k}} T_{k,j}$ and \\$\sum_{j=1}^{n_{k}}
     \phi_{i}(x_{k,j}) T_{k,j}$ live on $\operatorname{Ran} W_{k}$, it follows that
     \begin{equation}   \label{Tcon}
     \sum_{j=1}^{n_{k}} T_{k,j} = 0 \text{ and } \sum_{j=1}^{n_{k}} \phi_{i}(x_{k,j}) T_{k,j} = 0
     \text{ for each } i \text{ and } k.
     \end{equation}
     We next use that $\mu_{k} = \sum_{j=1}^{n_{k}} w_{j}^{(k)}
     \delta_{x_{k,j}}$ is a scalar extreme point:  it follows from
     Theorem \ref{T:extreme-scalar} that $0 = \sum_{j=1}^{n_{k}}
     w^{(k)}_{j} \bphi(x_{j})$ is an interior point of the convex
     hull of the vectors $\bphi(x_{k,1}), \dots, \bphi(x_{k,n_{k}})$
     in ${\mathbb R}^{m}$.  By criterion (3) in Proposition
     \ref{P:conv0}, the conditions \eqref{Tcon} applied entrywise
     now force that $T_{k,j} = 0$ for each $j=1, \dots, n_{k}$.  As
     $k$ is arbitrary, we have shown that $T_{k,j} = 0$ as required.

     \end{proof}

    The next result gives a partial converse to Theorem
    \ref{T:mu-form1}

    \begin{theorem} \label{T:mu-form1con}
	Suppose that $\mu \in \cC(X,N, \bphi)$ has the form 
	\eqref{mu-form'} such that (i) each $\mu_{k}$ is a scalar positive
	measure which is an extreme point for ${\mathcal C}^{1} = 
	\cC(X,1,\bphi)$
	(with supports not necessarily disjoint), (ii) each $W_{k}$ is
	positive semidefinite,  and (iii) the
	family of subspaces $\{ \operatorname{Ran} W_{k} \colon k=1,
	\dots, n\}$ is not weakly independent.  Then $\mu$ is not an
	extreme point of ${\mathcal C}$.
	 \end{theorem}
	
	 \begin{proof}
	     The assumption that $\{ \operatorname{Ran} W_{k} \colon
	     k = 1, \dots, n\}$ is {\em not} weakly independent means
	     that we can find a family of operators $\{T_{k} \colon
	     k=1, \dots, n\}$ on ${\mathbb C}^{N}$ such that $T_{k}
	     \ne 0$ for some $k$, $\sum_{k=1}^{N} T_{k} = 0$ and
	     $T_{k}$ lives on $\operatorname{Ran} W_{k}$ for each
	     $k$.  By considering either $\{\R T_{k} \colon k=1,
	     \dots, N\}$ or $\{\I T_{k} \colon k=1, \dots, N\}$, we
	     may suppose without loss of generality that each $T_{k}$
	     is complex Hermitian.  By choosing $\epsilon > 0$ but
	     sufficiently small, we can then arrange that $W_{k} \pm
	     \epsilon T_{k} \ge 0$ for each $k=1, \dots, n$.  We then define a
	     measure $\nu$ by
	     $$
	     \nu = \sum_{k=1}^{n} \epsilon T_{k} \mu_{k}.
	     $$
	     Then it is easily checked that $\nu \ne 0$, $\nu(\phi_{i}) = 0$,
	     $\nu(X) = 0$ and $\mu \pm \nu \ge 0$. As a consequence
	     of Lemma \ref{L:convex} it follows that $\mu$ is not an
	     extreme point of ${\mathcal C} = \cC(X, N, \bphi)$.
	
 \end{proof}

  We now present another concrete class of extreme points for a general cone
  ${\mathcal C}(N,X,\bphi)$ as in \eqref{cC}.

  \begin{theorem} \label{T:mu-form2} Suppose that $\mu \in [M(X)^{N
      \times N}]_{+}$ has the form
      \begin{equation}  \label{mu-spectral}
      \mu = \sum_{k-1}^{N} \mu_{k} P_{k}
      \end{equation}
      where $\{P_{1}, \dots, P_{N}\}$ is a pairwise-orthogonal family
      of rank-1 orthogonal projections summing to the identity
      operator $I$ on ${\mathbb C}^{N}$ and  where $\mu_{1}, \dots,
      \mu_{N}$ are scalar measures (not necessarily having disjoint
      supports and perhaps not even distinct) which are extremal for
      the cone of scalar measures ${\mathcal C}^{1} = \cC(X,1, \bphi)$.  Then $\mu$ is 
      extremal in the cone of matrix measures
      ${\mathcal C} = \cC(X,N,\bphi)$.
  \end{theorem}

  \begin{proof}
      Let $\mu$ be as in the statement of the Theorem and
      suppose that $\nu$ is a complex Hermitian $N\times N$ matrix
      measure with
      \begin{equation}  \label{nu-conditions'}
	  \nu(X) = 0, \quad \nu(\phi_{i}) = 0 \text{ for } i=1,
	  \dots, m, \quad \mu \pm \nu \ge 0.
 \end{equation}
 Factor the rank-1 orthogonal projection $P_{k}$ as $P_{k} = e_{k}
 e_{k}^{*}$ for a unit column vector $e_{k} \in {\mathbb C}^{N \times
 N}$.  Then we have
 $$
   0 \le P_{k} (\mu \pm \nu ) P_{k} = (\mu_{k} \pm \nu^{(1)}_{kk})
   P_{k}
 $$
 where $\nu_{kk}^{(1)}$ is the scalar measure given by
 $\nu_{kk}^{(1)} = e_{k}^{*} \nu e_{k}$.  The conditions
 \eqref{nu-conditions'} satisfied by $\nu$ imply that each
 $\nu_{kk}^{(1)}$ satisfies the conditions
 $$
 \nu_{kk}^{(1)}(X) = 0, \quad \nu_{kk}^{(1)}(\phi_{i}) = 0 \text{ for }
 i = 1, \dots, m, \quad \mu_{k} \pm \nu_{kk}^{(1)} \ge 0.
 $$
 Since $\mu_{k}$ is extremal for ${\mathcal C}^{1}$, a consequence of
 Lemma \ref{L:convex} is that $\nu_{kk}^{(1)} = 0$
 for each $k=1, \dots, n$.  It remains to show that the off-diagonal
 components of $\nu$ with respect to the orthonormal basis $\{e_{1},
 \dots, e_{N}\}$ are also all zero, i.e., $\nu_{k k'}^{(1)} =
 0$ for $1 \le k < k' \le N$, where $\nu_{k k'}^{(1)} = e_{i}^{*} \nu
 e_{k'}$.

 Toward this goal, we consider the $2 \times 2$ matrix measure
 $\left[ \begin{smallmatrix} e_{k}^{*} \\ e_{k'}^{*}
\end{smallmatrix} \right]( \mu \pm \nu) \left[ \begin{smallmatrix}
e_{k} & e_{k'} \end{smallmatrix} \right]$ which we can identify with
the $2 \times 2$ matrix measure
$$
\begin{bmatrix} \mu_{k} & \pm \nu_{k k'}^{(1)} \\
    \pm  \overline{ \nu_{k k'}^{(1)}} & \mu_{k'} \end{bmatrix}.
 $$
 From the fact that $\mu \pm \nu \ge 0$, we see that this $2 \times
 2$ matrix measure is positive for both choices of signs $\pm$.  It
 follows that necessarily $\operatorname{supp} \nu^{(1)}_{k k'}
 \subset \operatorname{supp} \mu_{k} \cap \operatorname{supp}
 \mu_{k'}$.  If $\{x_{k,1}, \dots, x_{k, n_{k}}\}$ is the support of
 $\mu_{k}$, then necessarily $\nu_{kk'}$ has the form
 $$
  \nu_{kk'}^{(1)} = \sum_{j=1}^{n_{k}} v^{kk'}_{j} \delta_{x_{k,j}}
 $$
 for some weights $v^{kk'}_{j}$ (where $v^{kk'}_{j} = 0$ whenever
 $x_{k,j} \in \operatorname{supp} \mu_{k}$ is not in
 $\operatorname{supp} \mu_{k'}$).  A consequence of the conditions
 \eqref{nu-conditions'} is that the set of conditions on the weights
 $\{v^{kk'}_{j} \colon j = 1, \dots, n_{k}\}$:
 \begin{equation}   \label{nukk'-sys}
 \sum_{j=1}^{n_{k}} v^{kk'}_{j} = 0, \quad \sum_{j-1}^{n_{k}}
 \phi_{i}(x_{k,j}) v^{kk'}_{j} = 0 \text{ for } i = 1, \dots, m.
 \end{equation}
 Since $\mu_{k} = \sum_{j=1}^{n_{k}} w^{(k)}_{j}
 \delta_{x_{k,j}}$ is extremal for ${\mathcal C}^{1}$, we know that
 $0 = \sum_{j=1}^{n_{k}} w^{(k)}_{j} \bphi(x_{k,j})$ is an
 interior point of the convex hull of $\{\bphi(x_{k,1}), \dots,
 \bphi(x_{k,n_{k}}\}$.  By criterion (3) in Proposition
 \ref{P:conv0}, conditions \eqref{nukk'-sys} then lead to the
 conclusion that
 $v^{kk'}_{j} = 0$ for $j = 1, \dots, n_{k}$. We conclude that
 $\nu^{(1)}_{kk'}$ is the zero measure.  Since the pair of indices
 $(k,k')$ is arbitrary, we now have that $\nu = 0$.  An application of
 Lemma \ref{L:convex} then tells us that $\mu$ is extremal for
 ${\mathcal C} = \cC(X,N,\bphi)$ as wanted.
  \end{proof}

  \begin{remark} \label{R:spectralmeasure} {\em By combining Theorems
      \ref{T:mu-form2} and \ref{T:extreme-matrix}, we see that any
      measure of the form \eqref{mu-spectral} must satisfy the
      conditions of Theorem \ref{T:extreme-matrix} when expressed in
      the form \eqref{mu-form''}.  There does not appear to be any
      obvious direct proof of this implication.
      }\end{remark}

  \begin{remark} \label{R:special}  {\em It is easily seen that the
      family of subspaces  $\{ \operatorname{Ran} P_{k} \colon 1 \le 
      k \le N\}$ is weakly independent whenever $\{P_{1}, \dots,
      P_{N}\}$ is a pairwise-orthogonal family of orthogonal
      projections on ${\mathbb C}^{N}$. Let us say that a $\mu \in
      {\mathcal C} = \cC(X,N, \bphi)$ (as in \eqref{cC}) is \textbf{special} if
      $\mu$ has a presentation of the form
      \begin{equation}  \label{special}
        \mu = \sum_{k=1}^{n} \mu_{k} W_{k}
      \end{equation}
      where each $\mu_{k}$ is a scalar positive measure extremal in 
      the cone of scalar measures
      ${\mathcal C}^{1}= \cC(X,1, \bphi)$ and where the family of subspaces 
      $\{\operatorname{Ran} W_{k} \colon k=1, \dots, n\}$ is weakly
      independent.  Thus the extremal measures identified in Theorem
      \ref{T:mu-form1} and those identified in Theorem
      \ref{T:mu-form2} are all special, but the class of special
      measures is more general the either of these special cases.
      For the special cases $N=1$ or $m=0$, we see that the set of
      extreme points $\partial_{e}{\mathcal C}$ consists exactly of
      the special measures.  More generally, in the examples which we have computed, 
      it turns out that special measures are extremal, but we do not 
      know if this is the case in general.  On the other hand there 
      are examples where there are extremal measures which are not 
      special (see Corollary \ref{C:special} below).
       }\end{remark}

\section{The Herglotz class over a finitely connected planar domain}  
\label{S:Herglotz}

In this section we let ${\mathcal R}$ denote a bounded  
domain (connected, open set) in the complex plane whose boundary 
consists of $m+1$ smooth Jordan curves; we refer to \cite{Fisher, Grunsky} as general 
references for the function theory on such domains.  We let $\partial_{0}, \partial_{1}, \dots, 
\partial_{m}$ denote the $m+1$ boundary components with $\partial_{0}$ 
denoting the boundary of the unbounded component of the complement 
${\mathbb C} \setminus {\mathcal R}$ of ${\mathcal R}$ in the complex 
plane.  For a fixed natural number $N$, we let ${\mathcal 
H}^{N}({\mathcal R})$ denote the set of (single-valued) holomorphic $N \times N$ 
matrix-valued functions $F(z)$ on ${\mathcal R}$  with  positive real 
part:  $\R F(z) \ge 0$ for $z \in {\mathcal R}$.  We often fix a 
point $t_{0} \in {\mathcal R}$ and consider ${\mathcal 
H}^{N}({\mathcal R})$ subject to the normalization $F(t_{0}) = 
I_{N}$; denote this normalized Herglotz class by 
${\mathcal H}^{N}({\mathcal R})_{I}$.  In case $N=1$ we write simply $\cH(\cR)_{1}$.

A standard normal families argument combined with the classical 
Harnack inequality shows that 
$\cH^{N}(\cR)_{I}$ is a compact convex subset of the locally convex 
topological space $\operatorname{Hol}(\cR)$ consisting of all 
holomorphic $N\times N$ matrix-valued functions on $\cR$ with the 
topology of uniform pointwise-convergent on compact subsets of 
$\cR$.  Our goal in this section is to apply the results of Section 
\ref{S:extreme} to characterize the extreme points of 
$\cH^{N}(\cR)_{I}$. 

\subsection{The scalar-valued Herglotz class over $\cR$}  
\label{S:scalarHerglotz}

For the scalar-valued case $N=1$, characterization of the extreme 
points of $\cH(\cR)_{1}$ is worked out in various 
places (see \cite{AHR, DM, Pick}).  The first step is to transform 
the problem to one of the form in Section \ref{S:extreme-scalar} as 
follows.  One can solve the Dirichlet problem on such domains: thus 
for given $u \in C_{{\mathbb R}}(\partial \cR)$, there is a unique 
function $u^{\wedge} \in C_{{\mathbb R}}(\cR^{-})$ so that
$u^{\wedge}|_{\cR}$ is harmonic on $\cR$ and
\begin{equation}   \label{harmext}
  u^{\wedge}|_{\partial R} = u.
\end{equation}
By the Riesz representation theorem, there is a Borel measure 
$\omega_{t_{0}}$ on $\partial \cR$ (the {\em harmonic measure for the 
fixed point } $t_{0}$) so that
$$
  u^{\wedge}(t_{0}) = \int_{\partial \cR} u(\zeta) \, {\tt d} 
  \omega_{t_{0}}(\zeta).
$$
It is known that the harmonic measure $d \omega_{z}$ for any other 
point $z \in \cR$ is mutually bounded absolutely continuous with 
respect to $d \omega_{t_{0}}$; hence there is a 
function $\cP_{z}(\cdot)$ on $\partial \cR$ (the {\em Poisson kernel}
for the region $\cR$ with the normalization that $\cP_{t_{0}}(\zeta) 
\equiv 1 $ on $\partial \cR$) so that
\eqref{harmext} becomes
\begin{equation}   \label{Poisson-rep}
  u^{\wedge}(z) = \int_{\partial \cR} u(\zeta) \cP_{z}(\zeta) \, {\tt 
  d}\omega_{t_{0}}(\zeta)
\end{equation}
This formula can be extended to measures in a natural way:  given a 
Borel measure $\mu$ on $\partial \cR$, define a function $u^{\wedge}$ 
on $\cR$ by
\begin{equation}   \label{measure-Poisson-rep}
\mu^{\wedge}(z) = \int_{\partial \cR} \cP_{z}(\zeta) \, {\tt d}\mu(\zeta).
\end{equation}
Note that if a continuous function $u$ on $\partial \cR$ is 
identified with the measure ${\tt d}\mu_{u}(\zeta) = u(\zeta) {\tt 
d}\omega_{t_{0}}(\zeta)$, then formula \eqref{measure-Poisson-rep} 
agrees with \eqref{Poisson-rep}.  Moreover, there is a converse for 
the case of positive harmonic functions: 
{\em any} positive harmonic function on $\cR$ is of the form 
$u^{\wedge}$ as in \eqref{measure-Poisson-rep} for a uniquely 
determined positive Borel measure $\mu$ on $\cR$.

One of the difficulties with function theory on multiply-connected 
domains (in contrast with function theory on the unit disk) is that 
harmonic functions need not have single-valued harmonic conjugates;
consequently,  a given harmonic function $u^{\wedge}$ on $\cR$ can 
fail to be the real part of any (single-valued) holomorphic function 
on $\cR$.  A natural task then is:  {\em  given a harmonic 
function $h$ on $\cR$ having the form $u^{\wedge}$ as in \eqref{Poisson-rep} or more generally
$\mu^{\wedge}$ as in
\eqref{measure-Poisson-rep}, describe in terms of the function $u \in 
C_{{\mathcal R}}(\partial \cR)$ (or in terms of the measure $\mu$ on 
$\partial \cR$) when is it the case that $h$ 
has a single-valued harmonic conjugate}.  Note that the second case 
covers the first case by putting $d\mu = u \, d\omega_{t_{0}}$ so it 
suffices to consider the second case.
The solution is quite elegant 
(see \cite{AHR, DM, Pick}) and can be described as follows.  One can 
show that there exists a set $\bphi =\{\phi_{1}, \dots, \phi_{m}\}$ of
$m$ continuous real-valued functions on $\partial \cR$
such that
\begin{align} 
   &  \bphi = \text{ real basis for } L^{2}(\omega_{t_{0}}) \ominus 
[H^{2}(\omega_{t_{0}}) \oplus \overline{H^{2}(\omega_{t_{0}})}],
\text{ or equivalently} \notag  \\
& \{ \phi_{1} {\tt d}\omega_{t_{0}}, \dots,
\phi_{d} {\tt d}\omega_{t_{0}}\} = \text{ real basis for }  (A(\cR) + 
\overline{A(\cR)})^{\perp}.
\label{basis}
\end{align}
Here $H^{2}(\omega_{t_{0}})$ is the Hardy space of analytic functions 
over $\cR$ based on the measure $\omega_{t_{0}}$ on $\partial \cR$ 
(see e.g.~\cite{Fisher}), the overline denotes complex conjugation, 
$A(R)$ is the algebra of continuous functions on $\cR^{-}$ (the 
closure of $\cR$) which are holomorphic on $\cR$, and the notation 
$\perp$ denotes the annihilator computed in the space of Borel 
measures $M(\partial \cR)$ on $\partial \cR$ dual to the Banach space $C(\partial \cR)$
of continuous functions on $\partial \cR$.
Then the result is:  {\em $\mu^{\wedge}$ given by 
\eqref{measure-Poisson-rep} has a single-valued harmonic conjugate 
$\widetilde \mu^{\wedge}$ if and only if the orthogonality conditions 
\begin{equation}   \label{orthogonality}
 \int_{\partial \cR} \phi_{i}(\zeta) \, {\tt d}\mu(\zeta) = 0 \text{ for 
 } i = 1, \dots, m.
\end{equation}
are satisfied.}  Moreover, the condition that $u^{\wedge}(t_{0}) = 1$ 
corresponds to the condition that $\mu(\partial \cR) = 1$ (so $\mu$ 
is a probability measure on $\partial \cR$).  {\em Throughout this 
section, the notation $\bphi =\{\phi_{1}, \dots, \phi_{m}\}$ refers to a fixed
$r$-tuple of real-valued continuous functions on $\partial \cR$ 
constructed as in \eqref{basis}.}

All these observations lead to a parametrization of scalar-valued 
normalized Herglotz class $\cH(\cR)_{1}$ as follows. 
Given a measure $\mu \in \cC( \partial \cR, 1, \bphi)$, define a positive
harmonic function $\mu^{\wedge}$ on 
$\partial \cR$ by \eqref{measure-Poisson-rep}.  Since $\mu \in \cC(
\partial \cR, 1, \bphi)$, $\mu$ satisfies the 
orthogonality conditions \eqref{orthogonality} and hence any harmonic 
conjugate  of 
$\mu^{\wedge}$ is single-valued.   Then there is a unique  such 
harmonic conjugate $\widetilde \mu^{\wedge}$ so that $\widetilde \mu^{\wedge}(t_{0}) = 0$. 
If we set $f_{\mu}(z) = \mu^{\wedge}(z) + i \widetilde \mu^{\wedge}(z)$, then
$f_{\mu} \in \cH(\cR)_{1}$. Furthermore, any $f \in  
\cH(\cR)_{1}$ has the form $f_{\mu}$ for a uniquely determined 
$\mu \in \cC(\partial \cR, 1, \bphi)$.
Thus {\em there is a one-to-one correspondence between the normalized 
Herglotz class $\cH(\partial \cR)_{1}$ and the convex set of 
probability measures $\cC(\partial \cR, 1, \bphi)$}.
As the correspondence is affine, we can also say:  {\em the function $f$ 
is extremal for the compact convex set 
$\cH_{1}(\cR)$ if and only if $f = f_{\mu}$ where $\mu$ is an 
extremal measure for the compact convex set $\cC(\partial \cR, 1, 
\bphi)$. } 

Thus to describe the set of extreme points for $\cH(\cR)_{1}$, it 
suffices to describe the extreme points of $\cC(\partial \cR, 1, \bphi)$, 
exactly a problem analyzed in Theorem 
\ref{T:extreme-scalar} above.  However, for this function-theory 
context, much more definitive detailed information is available.

\begin{theorem}  \label{T:extreme-scalar-Her} (See \cite[Theorem 
    1.3.17]{AHR}, \cite[Lemma 2.10]{DM}, \cite[Lemma 3.7]{Pick}.)
   For any $m+1$-tuple $\bx = (x_{0}, \dots, x_{m})$ of points on $\partial 
    \cR$ such that $x_{j} \in \partial_{j}$ for each $j=0,1, \dots, 
    m$, there is a unique set of positive weights $w^{\bx}_{0}, 
    \dots, w^{\bx}_{m}$ summing up to $1$ such that the measure $\mu$ 
    given by
    \begin{equation}  \label{mu-form-scalarHer}
    \mu = w^{\bx}_{0} \delta_{x_{0}} + \cdots w^{\bx}_{m} \delta_{x_{m}}
    \end{equation}
    is extremal for $\cC(1, \partial \cR, \phi_{1}, \dots, 
    \phi_{m})$, and, conversely, {\em any} extremal measure  $\mu$ 
    for $\cC( \partial \cR, 1, \bphi)$ arises in 
    this way.
\end{theorem}

From the point of view of Theorem \ref{T:extreme-scalar}, the added 
content of Theorem \ref{T:extreme-scalar-Her} is as follows.  
For the case where $X = \partial \cR$ and $\bphi = \{\phi_{1}, \dots, 
\phi_{m}\}$ 
is constructed as in \eqref{basis}, then {\em the $n$-tuple of points $x_{1}, \dots, 
x_{n}$ in $\partial \cR$ is such that the zero vector 
$0 \in {\mathbb R}^{m}$ is an interior point of the convex hull of the set of vectors
$$
  \left\{ \begin{bmatrix} \phi_{1}(x_{1}) \\ \vdots \\ 
  \phi_{m}(x_{1}) \end{bmatrix}, \dots,  \begin{bmatrix} \phi_{1}(x_{n}) \\ \vdots \\ 
  \phi_{m}(x_{n}) \end{bmatrix} \right\} \subset {\mathbb R}^{m}
$$
if and only if $n = m+1$ and the $m+1$-tuple now indexed as $x_{0}, 
x_{1}, \dots, x_{m}$ consists of exactly one point from each 
boundary component $\partial_{j} \subset \partial \cR$ of $\cR$.}

\smallskip

Following \cite{AHR}, let us introduce the notation
\begin{equation}   \label{Rtorus}
  {\mathbb T}_{\cR} = \partial_{0} \times \cdots \times \partial_{m}
 \end{equation}
for the Cartesian product of the boundary components of $\cR$; we 
think of this as the ``$\cR$-torus'' since it plays the same role in 
integral representation formulas for Herglotz functions over $\cR$ as 
does the usual torus ${\mathbb T} = \partial {\mathbb D}$ for 
integral representation formulas for Herglotz functions over the unit 
disk ${\mathbb D}$. 
Theorem \ref{T:extreme-scalar-Her} provides a ${\mathbb 
T}_{\cR}$-parametrization of the extreme points of $\cH(\cR)_{1}$ 
as follows.  For a given $\bx = (x_{0}, x_{1}, \dots, x_{m}) \in {\mathbb T}_{\cR}$, 
we let $\mu_{\bx}$ be the extremal measure of $\cC(\partial \cR,  1,
\bphi)$ given by \eqref{mu-form-scalarHer}. Given 
any $\bx \in {\mathbb T}_{\cR}$, we let $f_{\bx}$ be the unique 
holomorphic function on $\cR$ determined by
\begin{equation}   \label{fbx}
  \R f_{\bx}(z) = \int_{\partial \cR} \cP_{z}(\zeta) {\tt 
  d}\mu_{\bx}(\zeta), \quad \I f_{\bx}(t_{0}) = 0.
\end{equation}
Then {\em the extreme points of normalized Herglotz functions 
$\cH(\cR)_{1}$ consist exactly of the functions $f_{\bx}$ with 
$\bx \in {\mathbb T}_{\cR}$.}

We next observe that the convex set $\cC( \partial \cR, 1, \bphi)$ is 
compact and convex in the space of real Borel measures 
$M(\partial \cR)$ over $\cR$ where the latter space carries the 
weak-$*$ topology.  As $M(\partial \cR)$ is the dual of $C_{{\mathbb 
R}}(\partial R)$ which is a separable Banach space, it follows that 
the unit ball in $M(\partial \cR)$ is metrizable. Hence the second 
statement in Theorem \ref{T:Choquet} applies. 
 Furthermore, when we use the 
correspondence between $\partial_{e} \cC(\partial R, 1, \phi)$
 and ${\mathbb T}_{\cR}$ to transport the weak-$*$ 
topology on  $\partial_{e} \cC(\partial R, 1, \bphi)$ to a 
topology on ${\mathbb T}_{\cR}$, one can check 
that the topology so obtained is just the Euclidean topology on 
${\mathbb T}_{\cR}$ inherited as a subset of ${\mathbb C}^{m+1}$. 
Thus, by Theorem \ref{T:Choquet} 
above, any measure $\mu \in \cC( \partial \cR, 1, \bphi)$ has an integral representation
$$
  \mu = \int_{{\mathbb T}_{\cR}} \mu_{\bx} {\tt d}\nu(\bx)
$$
where the integral is defined in the weak sense:
\begin{equation}   \label{weak-int}
  \int_{\partial \cR} \phi(\zeta) {\tt d}\mu(\zeta) = 
  \int_{{\mathbb T}_{\cR}} \left[ \int_{\partial \cR} \phi(\zeta) {\tt d} 
  \mu_{\bx}(\zeta) \right] {\tt d}\nu(\bx) \text{ for each } \phi \in 
  C_{{\mathbb R}}(\partial \cR).
\end{equation}
This leads to the following integral representation formula for 
functions $f$ in the normalized Herglotz class $\cH^{1}_{1}(\cR)$.

\begin{theorem}  \label{T:scalarHer-int-rep} (See \cite[Theorem 
    1.3.26]{AHR}.)
Given $f \in \cH(\cR)_{1}$, there is a probability measure $\nu$ 
on ${\mathbb T}_{\cR}$ so that
\begin{equation}   \label{scalarHer-int-rep}
    f(z) =  \int_{{\mathbb T}_{\cR}} f_{\bx}(z)\, {\tt d}\nu(\bx).
 \end{equation}
 \end{theorem}
 
 \begin{proof}
We have already seen that any $f \in \cH(\cR)_{1}$ is associated 
with a  uniquely determined measure $\mu \in \cC( \partial \cR, 1, \bphi)$ so that  
$\R f = \mu^{\wedge}$ as 
in \eqref{measure-Poisson-rep}.
  Plugging this $\mu$ into 
\eqref{weak-int} (and setting $f(z) = \cP_{z}(\zeta)$) tells us that there is a probability measure $\nu$ 
on ${\mathbb T}_{\cR}$ so that
\begin{align*}
\R f(z) & = \int_{\partial \cR} \cP_{z}(\zeta) {\tt d}\mu(\zeta)  \\
& =\int_{{\mathbb T}_{\cR}} \left[ \int_{\partial \cR} \cP_{z}(\zeta) 
{\tt d}\mu_{\bx}(\zeta) \right]\,  {\tt d}\nu(\bx) \\   
& = \int_{{\mathbb T}_{\cR}} \R f_{\bx}(z)\, {\tt d}\nu(\bx).
\end{align*}
By the uniqueness of the harmonic conjugate normalized to have value 
$0$ at $t_{0}$, the formula \eqref{scalarHer-int-rep} now follows.
\end{proof}

\subsection{The matrix-valued Herglotz class of $\cR$} 
\label{S:matrixHerglotz}

We now wish to obtain results parallel to Theorem 
\ref{T:extreme-scalar-Her} and Theorem \ref{T:scalarHer-int-rep} for 
the normalized matrix-valued Herglotz class $\cH^{N}(\cR)_{I}$.  For 
the matrix-valued case the function theory is not as highly developed.
The implication is that the results which we do obtain are not as 
explicit as for the scalar-valued case.

By applying Theorem \ref{T:extreme-matrix} to the convex set 
$\cC(\partial \cR, N, \bphi)$ (where $\bphi = \{\phi_{1}, \dots, 
\phi_{m}\}$ as usual is as in \eqref{basis}), we get most of the 
following somewhat less  
explicit analogue of Theorem \ref{T:extreme-scalar-Her}.

\begin{theorem} \label{T:extreme-Her}  Consider the convex set 
    $\cC(\partial \cR, N, \bphi)$.
   Then  the
   $N \times N$ matrix-valued Borel measure $\mu$ on $\partial \cR$
   is extremal for $\cC(\partial \cR, N, \bphi)$ 
   if and only if there is a 
 natural number $n$ with $1 \le n \le (m+1)N^{2}$, $n$
	    distinct points $x_{1}, \dots, x_{n}$ in $\partial \cR$ 
	   together with $N \times N$ matrix weights $W_{1}, \dots, W_{n}$ such that
$ W_{j} \ge 0$ for each  $j$, 
$\sum_{j=1}^{n} W_{j} = I_{N}$, and $ \{\operatorname{Ran} W_{j} 
\colon j = 1, \dots, n\}$ is $\bphi(\bx)$-constrained weakly independent 
(see Definition \ref{D:inthull}) where we set 
$$
\bphi(\bx) = \left\{ \begin{bmatrix} \phi_{1}(x_{1}) \\ \vdots \\ 
\phi_{m}(x_{1}) \end{bmatrix}, \dots,  \begin{bmatrix} \phi_{1}(x_{n}) \\ \vdots \\ 
\phi_{m}(x_{n}) \end{bmatrix}  \right\} \subset {\mathbb R}^{m},
$$
so that $\mu$ has the representation 
\begin{equation}   \label{mu-rep-Her}
\mu = \sum_{j=1}^{n} W_{j} \delta_{x_{j}}.
\end{equation}

Conversely, if $\mu$ of the form \eqref{mu-rep-Her} is extremal for 
$\cC(\partial \cR, N, \bphi)$, then the additional property
\begin{equation}  \label{bdryWinv}
    \sum_{j \colon x_{j} \in \partial_{r}} W_{j} \text{ is invertible 
    for each } r =  0,1, \dots, m
 \end{equation}
 holds;  consequently the number of points $n$ in the support 
 $\operatorname{supp} \mu$ of $\mu$ in fact satisfies $m+1 \le n$ 
 with at least one point $x_{j}$ from the support of $\mu$ in each 
 connected component $\partial_{r}$ of $\partial \cR$.
	  \end{theorem}
	  
	  \begin{proof}  What is added here going beyond the 
	      structure given by Theorem \ref{T:extreme-matrix} for 
	      the general (non function-theoretic) case is the 
	      information on the converse direction given by 
	      \eqref{bdryWinv}.  To see this, we note that for any 
	      unit vector $\bu \in {\mathbb C}^{N}$, the 
	      scalar-valued measure $\mu_{\bu}(\Delta) : =  \bu^{*} 
	      \mu(\Delta) \bu$ is in the convex set of scalar measures 
	      $\cC(\partial \cR, 1, \phi)$.  We now quote the result of Lemma 
	      1.3.1 in \cite{AHR}: {\em if $\mu^{(1)}$ is 
	      a nonzero scalar positive measure on $\partial \cR$ 
	      such that $\int_{\partial \cR} \phi_{i} {\tt d}\mu^{(1)} = 0$ 
	      for $i=1, \dots, m$} (so $\mu^{(1) \wedge}$ has a 
	      single-valued harmonic conjugate), {\em then 
	      $\mu^{(1)}(\partial_{r}) > 0$ for each $r=0,1, \dots, m$}.
	      Consequently, if $\sum_{j \colon x_{j} \in 
	      \partial_{r}} W_{j} = \mu(\partial_{r})$ is singular 
	      for some $r$, then there is a unit vector $\bu$ so that 
	      $\mu_{\bu}(\partial_{r}) = u^{*} \mu(\partial_{r}) u = 
	      0$.  From \cite[Lemma 1.3.1]{AHR} we conclude that 
	      $\bu^{*} \mu \bu$ is the zero measure, in contradiction 
	      with $\mu(\partial \cR) = I_{N}$.
	\end{proof}

\begin{remark}  \label{R:lessexplicit}
   {\em  We note that what is lacking in Theorem \ref{T:extreme-Her} (as 
    compared to Theorem \ref{T:extreme-scalar-Her}) is an explicit 
    characterization as to which natural numbers $n$ between $m+1$ 
    and $(m+1)N^{2}$ and which associated $n$-tuples of points $\bx = 
    (x_{1}, \dots, x_{n})$ actually arise in a representation 
    \eqref{mu-rep-Her} for an extreme point of $\cC(\partial \cR, N, 
    \bphi)$, beyond the information that $\bx$ must include at least 
    one point from each boundary component $\partial_{1}, \dots, 
    \partial_{m}$.   
    Also it is not clear to what extent the $n$-tuple of 
    points $\bx$ determines the associated $n$-tuple of matrix 
    weights $W_{1}, \dots, W_{n}$; note that the classes of examples 
    from Theorems \ref{T:mu-form1} and \ref{T:mu-form2} show that it 
    is certainly not the case that the support $\bx = \{x_{1}, \dots, 
    x_{n}\}$ uniquely determines the associated set of matrix 
    weights $\bw = \{W_{1}, \dots, W_{n}\}$.
    } \end{remark}
    
Despite the lack of explicitness in the characterization of the 
extreme points of $\cC(\partial \cR, N, \bphi)$ as explained in 
Remark \ref{R:lessexplicit},
 we can still pursue the matrix analogue of much of the analysis done for 
 the scalar-valued normalized Herglotz class as follows.
 
 It is straightforward to see that positive $N \times N$ matrix-valued 
 harmonic functions $H$ are given via a Poisson representation
 $$
  H(z) = \mu^{\wedge}(z) : = \int_{\partial \cR} \cP_{z}(\zeta) {\tt 
  d}\mu(\zeta)
 $$
 where now $\mu$ is a complex Hermitian positive $N \times N$ 
 matrix-valued measure on $\partial \cR$.  Moreover, the harmonic 
 matrix-valued function $\mu^{\wedge}(z)$ has a single-valued 
 matrix-valued harmonic conjugate if and only if the matrix measure 
 $\mu$ satisfies the orthogonality conditions \eqref{orthogonality}
 (with respect to the {\em scalar-valued} functions $\bphi: = \{ 
 \phi_{1}, \dots, \phi_{m}\}$, and the normalization condition that 
 $\mu^{\wedge}(t_{0}) = I_{N}$ translates to the condition on $\mu$ 
 that $\mu(X) = I_{N}$. By continuing an analysis parallel to what was 
 done above for the scalar-valued case, we arrive at the following:
 {\em Given $\mu \in \cC(\partial \cR, N, \bphi)$, there is a unique 
 $F_{\mu} \in \cH^{N}(\cR)_{I}$ so that 
 \begin{equation}   \label{Herrep}
   \R F_{\mu}(z) = \int_{\partial \cR}\cP_{z}(\zeta){\tt 
   d}\mu(\zeta).
  \end{equation}
  Conversely, any $F \in \cH^{N}(\cR)_{I}$ arises in this way 
  from a $\mu \in \cC(\partial \cR, N, \bphi)$.  Moreover, the 
  function $F_{\mu}$ is extremal in $\cH^{N}(\cR)_{I}$ if and 
  only if $\mu$ is extremal in $\cC( \partial \cR, N, \bphi)$.}

 The set of complex Hermitian $N \times N$ matrix-valued measures $[M(X)^{N \times 
 N}]_{h}$ is the dual space of the separable real Banach space 
 $[C(\partial \cR)^{N \times N}]_{h}$ of complex Hermitian $N\times N$ 
 matrix-valued continuous functions on $\partial \cR$, and hence the 
 unit ball is metrizable and the second statement in Theorem 
 \ref{T:Choquet} applies.  The set of extreme points 
 $\partial_{e}\cC(\partial \cR,N, \bphi)$ is a 
 Borel subset of $[M(X)^{N \times N}]_{h}$ and we have a (admittedly 
 somewhat implicit) parametrization from Theorem \ref{T:extreme-Her}.
 In detail, let us denote by ${\mathbb T}^{N}_{\cR}$ (the matrix 
 $\cR$-torus) the set 
 \begin{align*}
  {\mathbb T}^{N}_{\cR} = &  \{ (\bx, \bw) \colon \bx = (x_{1}, \dots, 
  x_{n}) \subset \partial \cR \text{ and } \bw = (W_{1}, \dots, 
  W_{n}) \subset [{\mathbb C}^{N \times N}]_{+} \\
  & \quad \text{ are as in Theorem \ref{T:extreme-Her}} \}.
 \end{align*}
 Given $(\bx, \bw) \in {\mathbb T}^{N}_{\cR}$, there is an associated 
 extremal measure $\mu_{\bx, \bw} \in \partial_{e} \cC( \partial 
 \cR, N, \bphi)$ given by \eqref{mu-rep-Her} in Theorem 
 \ref{T:extreme-Her}, and all extremal measures $\mu \in \partial_{e} 
 \cC(\partial \cR, N, \bphi)$ are of the form $\mu_{\bx, \bw}$ for 
 some $(\bx, \bw) \in {\mathbb T}^{N}_{\cR}$.  We topologize ${\mathbb T}^{N}_{\cR}$ by 
 transporting the weak-$*$ topology on the the associated set of 
 measures $\mu_{\bx, \bw}$ for $(\bx, \bw) \in {\mathbb T}^{N}_{\cR}$.
 Then by Theorem \ref{T:Choquet}, given any $\mu \in 
 \cC(\partial \cR,N, \bphi)$, there is a Borel probability measure $\nu$ on 
 ${\mathbb T}^{N}_{\cR}$ so that
 $$
   \mu = \int_{{\mathbb T}^{N}_{\cR}} \mu_{\bx, \bw}\, {\tt 
   d}\nu(\bx, \bw)
 $$
 with the integral interpreted in the weak sense:
 \begin{equation}  \label{matrix-weak-int}
\operatorname{tr} \left( \int_{\partial \cR} \Phi(\zeta) \, {\tt 
d}\mu(\zeta) \right) = 
\int_{{\mathbb T}^{N}_{\cR}} \operatorname{tr}
 \left( \int_{\partial \cR} \Phi(\zeta) \, {\tt 
d}\mu_{\bx, \bw}(\zeta) \right) \, {\tt d}\nu(\bx, \bw)
\end{equation}
for any $\Phi \in [C(\partial \cR)^{N \times N}]_{h}$.

Since extreme points of $\cH^{N}(\cR)_{I}$ correspond to extreme 
points of $\cC(\partial \cR, N, \bphi)$ in accordance with the formula
\eqref{Herrep}, we see that the extreme points of the normalized 
Herglotz class $\cH^{N}(\cR)_{I}$ are exactly the functions 
$F_{\bx, \bw}$ determined by
\begin{equation}   \label{extreme-Her}
 \R F_{\bx, \bw}(z) = \int_{\partial \cR} \cP_{z}(\zeta)\, {\tt 
 d}\mu_{\bx, \bw}(\zeta), \quad \I  F_{\bx, \bw}(t_{0}) = 0.
\end{equation}

We are now led to the matrix-valued analogue of Theorem 
\ref{T:scalarHer-int-rep}.

\begin{theorem} \label{T:Her-int-rep}
    Given $F \in \cH^{N}(\cR)_{I}$, there is a probability 
    measure $\nu$ on ${\mathbb T}^{N}_{\cR}$ so that
  \begin{equation}   \label{Her-int-rep}
      F(z) = \int_{{\mathbb T}^{N}_{\cR}} F_{\bx, \bw}(z) \, {\tt 
      d}\nu(\bx, \bw).
   \end{equation}
    \end{theorem}
    
    \begin{proof}
	We have noted that any $F \in \cH^{N}(\cR)_{I}$ is 
	associated with a measure $\mu \in \cC(\partial \cR, N,
	\bphi)$ as in \eqref{Herrep}.  For $X$ an arbitrary complex 
	Hermitian $N \times N$ matrix, we use the representation 
	\eqref{matrix-weak-int} with $\Phi(\zeta) = X \cP_{z}(\zeta)$ 
	to get
 \begin{align*}
     \operatorname{tr} ( X \R F(z)) & = 
     \operatorname{tr} \left( \int_{\partial \cR} X \cP_{z}(\zeta) 
     {\tt d}\mu(\zeta) \right) \\
     & = \int_{{\mathbb T}^{N}_{\cR}} \operatorname{tr} \left( 
     \int_{\partial \cR} X \cP_{z}(\zeta) \, {\tt d}\mu_{\bx, 
     \bw}(\zeta) \right) d\nu(\bx, \bw)  \\
      & = \int_{{\mathbb T}^{N}_{\cR}}  \operatorname{tr} \left( X \R F_{\bx, 
     \bw}(z) \right) \, {\tt d}\nu(\bx, \bw)  \\
     & = \operatorname{tr} \left( X \int_{{\mathbb T}^{N}_{\cR}} 
    \R  F_{\bx, \bw}(z)\, {\tt d}\nu(\bx, \bw) \right).
     \end{align*}
    Since $X \in [{\mathbb C}^{N \times N}]_{h}$ is arbitrary, we 
    conclude that
    $$
      \R F(z) = \int_{{\mathbb T}^{N}_{\cR}} \R F_{\bx,\bw}(z) \, {\tt 
      d}\nu(\bx, \bw).
     $$
   By uniqueness of harmonic conjugate with value $0$ at $t_{0}$, the 
    representation \eqref{Her-int-rep} now follows.
	\end{proof}

	\section{The Schur class over a finitely connected planar 
	domain}  \label{S:Schur}
	
	We define the (strict) {\em $N\times N$-matrix Schur class} over 
	$\cR$,  denoted by $\cS^{N}(\cR)$, to be the class of all 
	holomorphic functions on $\cR$ with values equal to $N \times 
	N$ matrices such that $\| S(z) \| < 1$ for $z \in \cR$.  The 
	normalized strict Schur class $\cS^{N}(\cR)_{0}$ consists of such functions $S$ such 
	that in addition $S(t_{0}) = 0$. A consequence of the Schwarz 
	lemma is that a holomorphic function $S$ with (not necessarily 
	strict) contraction values and with $S(t_{0}) = 0$ 
	necessarily has strictly contractive values on all of $\cR$.
	
	The classes 
	$\cH^{N}(\cR)_{I}$ and $\cS^{N}(\cR)_{0}$ correspond via a 
	linear-fractional change of variable, as summarized in the 
	following Proposition. We include the elementary proof since we shall 
	make use of the formulas in subsequent proofs.

	\begin{proposition}   \label{P:SchurHer}
	    The normalized Schur class $\cS^{N}(\cR)_{0}$ and 
	    the normalized Herglotz class $\cH^{N}(\cR)_{I}$ are 
	    related according to the following linear-fractional 
	    change-of-variable formulas:
	    \begin{align}
		& S \in \cS^{N}(\cR)_{0} \Leftrightarrow F:= (I - S)^{-1}(I + S) 
		 \in \cH^{N}(\cR)_{I}, \label{StoF} \\
	& F \in \cH^{N}(\cR)_{I} \Leftrightarrow S: = (F + I)^{-1} (F 
	- I) \in \cS^{N}(\cR)_{0}.  \label{FtoS}
	\end{align}
	Moreover, the transformations in \eqref{StoF} and 
	\eqref{FtoS} are inverse to each other.
	\end{proposition}
	
	\begin{proof} If $S \in \cS^{N}(\cR)_{0}$ and $F$ is defined 
	    as in \eqref{StoF}, then clearly $F(t_{0}) = I_{N}$ and
	  \begin{align}
	  &   F(z) + F(w)^{*}  = (I - S(z))^{-1}(I + S(z)) + (I 
	     +S(w)^{*}) (I - S(w)^{*})^{-1}  \notag \\
	     & \quad = (I - S(z))^{-1} [ (I + S(z)) (I - S(w)^{*}) + (I - 
	     S(z)) (I + S(w)^{*}) ] (I - S(w))^{-1} \notag \\
	     & \quad = 2 (I - S(z))^{-1} (I - S(z) S(w)^{*} ) (I - 
	     S(w)^{*})^{-1}
	     \label{F-defect}
	     \end{align}
	  from which we see that $\R F(z)$ is positive for $z \in 
	  \cR$ so $F \in \cH^{N}(\cR)_{I}$.
	  
	  Similarly, if $F \in \cH^{N}(\cR)_{I}$ and $S$ is defined 
	  as in \eqref{FtoS}, then clearly $S(t_{0}) = 0$ and
	  \begin{align}
	    &  I - S(z) S(w)^{*}  = I - (F(z) + I)^{-1} (F(z) - I) 
	      (F(w)^{*} - I) (F(w)^{*} + I)^{-1} \notag \\
	   & \quad  = (F(z) + I)^{-1} [ (F(z) + I) (F(w)^{*} + I) - (F(z) 
	   -I) (F(w)^{*} - I) ] (F(w)^{*} + I)^{-1} \notag \\
	   & \quad = 2 (F(z) + I)^{-1} [ F(z) + F(w)^{*} ] (F(w)^{*} + I)^{-1}
	   \label{S-defect}
	   \end{align}
	   from which we see that $S(z)$ is constrictive for $z \in 
	   \cR$ and hence $S \in \cS^{N}(\cR)_{0}$.
	   
	   Another formula which will prove useful later is
	   \begin{equation} \label{useful}
	       (F(z) + I)^{-1} = \frac{1}{2} (I - S(z))
	    \end{equation}
	    whenever $F$ and $S$ are related as in \eqref{StoF}.
	   
	   We leave to the reader the verification of the fact that the formulas 
	   \eqref{StoF} and \eqref{FtoS} are inverse to each other.
	     \end{proof}

	More generally, if $S$ is not in the normalized Schur class 
	but is in the strict (unnormalized) Schur class, we can apply 
	a matrix linear-fractional map mapping the unit ball of $N 
	\times N$ matrices to itself to obtain a new $S'$ which is in 
	the normalized Schur class.  Indeed, given any strictly 
	contractive $N \times N$ matrix $W$, the matrix linear 
	fractional map given by
	\begin{equation}   \label{LW}
	L_{W} \colon Z \mapsto [AZ + B] [CZ + D]^{-1}
	\end{equation}
	where 
	\begin{equation}   \label{ABCD}
	\begin{bmatrix} A & B \\ C & D \end{bmatrix} = 
	    \begin{bmatrix} (D_{W^{*}})^{-1} & -(D_{W^{*}})^{-1} W \\
		-W^{*} (D_{W^{*}})^{-1} & (D_{W})^{-1}  \end{bmatrix}
	\end{equation}
	where $D_{W} = (I - W^{*} W)^{\frac{1}{2}}$ and $D_{W^{*}} = 
	( I - W W^{*})^{\frac{1}{2}}$ denote the invertible defect 
	operators of $W$ and $W^{*}$, maps the open unit ball $\cB 
	{\mathbb C}^{N \times N} =  \{ Z \in {\mathbb C}^{N \times N} 
	\colon \| Z \| < 1\}$ biholomorphically to itself and maps 
	the given strict contraction matrix $W$ to $0$ (these 
	constructions go back at least to the paper of Phillips \cite{Phillips}):
	$$
	L_{W}[W] = 0.
	$$
        One can check that the
	linear-fractional map $(L_{W})^{-1}$ mapping $0$ back to $W$ 
	if given by
	$$   
	(L_{W})^{-1} \colon Z' \mapsto (A - Z' C)^{-1} (B - Z'D)
	$$
	with $A,B,C,D$ as in \eqref{ABCD}, or explicitly
	   \begin{align}  
		L_{W}^{-1} \colon Z' & \mapsto 
		\left( (D_{W^{*}})^{-1} + Z' W^{*} (D_{W^{*}})^{-1} 
		\right)^{-1} \left(Z' (D_{W})^{-1} + (D_{W^{*}})^{-1} 
		W \right) \notag  \\
		& = D_{W^{*}} (I + Z' W^{*})^{-1} (Z' + W) (D_{W})^{-1}
		 \label{LWinv}
		 \end{align}
	(where we made use of the intertwining relation
	$(D_{W^{*}})^{-1}W = W (D_{W})^{-1}$).  This formula will 
	prove useful below.

	Notice that in the scalar case with $w$ a point in the unit 
	disk, the matrix linear-fractional map $L_{W}$ simplifies to 
	the familiar M\"obius transformation
	$$
	L_{w} \colon z \mapsto (z-w)(1 - z \overline{w})^{-1}
	$$
	mapping the unit disk onto itself with the point $w \in 
	{\mathbb D}$ mapping to $0$.

	With these observations in hand, the following is immediate.
	
	\begin{proposition}   \label{P:normalize}
	    If the matrix function $S$ is in the strict Schur class 
	    $\cS^{N}(\cR)$,
	    then
	    $\widetilde S$ given by
	    \begin{equation}   \label{StotildeS}
	      \widetilde S(z) = L_{S(0)}[S(z)] \text{ with } L_{S(0)} 
	      \text{ given as in  \eqref{LW} and \eqref{ABCD}}
	     \end{equation}
	  is in the normalized Schur class $\cS^{N}(\cR)_{0}$.
	  \end{proposition}
	  
	  The following formula for the defect of $S$ in terms of the 
	  defect of $\widetilde S$ will also be useful below.
	  
	  \begin{proposition}   \label{P:Sdefect}
	      Suppose $S \in \cS^{N}(\cR)$ and $\widetilde S \in 
	      \cS^{N}(\cR)_{0}$ are related as in Proposition 
	      \ref{P:normalize}.  Then we have
	      \begin{equation}   \label{Sdefect}
		  I - S(z) S(w)^{*} = D_{S(0)^{*}} (I + \widetilde 
		  S(z) S(0)^{*})^{-1} ( I - \widetilde S(z) 
		  \widetilde S(w)^{*} ) (I + S(0) \widetilde 
		  S(w)^{*})^{-1} D_{S(0)^{*}}.
	\end{equation}
	 \end{proposition}
	 
	 \begin{proof}
	     From the representation \eqref{StotildeS} for 
	     $\widetilde S$ in terms of $S$, we solve for $S$ to get
	     $$
	      S(z) = (L_{S(0)})^{-1}[\widetilde S(z)].
	    $$
	    We now use the explicit formula for $(L_{S(0)})^{-1}$ 
	    determined from equation \eqref{LWinv} to get
	    $$
	    S(z) = D_{S(0)^{*}} (I + \widetilde S(z) S(0)^{*})^{-1} 
	    (\widetilde S(z) + S(0)) (D_{S(0)})^{-1}.
	    $$
	    Hence we get
	    \begin{align}
	&	I - S(z) S(w)^{*}  = 
	I - \left(L_{S(0)^{*}}\right)^{-1}[\widetilde S(z) ] 
\left(	\left(L_{S(0)}\right)^{-1}[\widetilde S(w)]  
\right)^{*}  \notag \\
& = I - D_{S(0)^{*}} (I + \widetilde S(z) S(0)^{*})^{-1} (\widetilde 
S(z) + S(0)) (D_{S(0)})^{-1} \cdot  \notag \\
& \quad \cdot (D_{S(0)})^{-1} (\widetilde S(w)^{*} + 
S(0)^{*}) (I + S(0) \widetilde S(w)^{*})^{-1} D_{S(0)^{*}} \notag  \\
 & = D_{S(0)^{*}} (I + \widetilde S(z) S(0)^{*})^{-1} X
(I + S(0) \widetilde S(w)^{*})^{-1} D_{S(0)^{*}}
\label{Sdefect1}
	\end{align}
	where we set 
 \begin{align}  
     X = & [I + \widetilde S(z) S(0)^{*}] (D_{S(0)^{*}})^{-2} [I + S(0) 
     \widetilde S(w)^{*}]  \notag \\
     & \quad -
     \left[\widetilde S(z) + S(0)\right] (D_{S(0)})^{-2} \left[\widetilde S(w)^{*} + 
     S(0)^{*}\right].
     \label{X}
   \end{align}
   In the computation to follow we use the intertwining relations
   \begin{equation}   \label{intertwine12}
       S(0)^{*} (D_{S(0)^{*}})^{-2} =  (D_{S(0)})^{-2} S(0)^{*}, \quad
       (D_{S(0)^{*}})^{-2} S(0) = S(0) (D_{S(0)})^{-2}.
       \end{equation}
      We now pick up the computation of $X$ in \eqref{X}:
   \begin{align*}
       X & = (D_{S(0)^{*}})^{-2} + \widetilde S(z) S(0)^{*} 
       (D_{S(0)^{*}})^{-2}  \\
       & \quad + (D_{S(0)^{*}})^{-2} S(0) \widetilde 
       S(w)^{*} + \widetilde S(z) S(0)^{*} (D_{S(0)^{*}})^{-2} S(0) 
       \widetilde S(w)^{*} \\
      & \quad  - \widetilde S(z) (D_{S(0)})^{-2} \widetilde S(w)^{*} - S(0) 
       (D_{S(0)})^{-2} \widetilde S(w)^{*} \\
       & \quad -\widetilde S(z) (D_{S(0)})^{-2} S(0)^{*} - S(0) 
       (D_{S(0)})^{-2} S(0)^{*} \\
       & = \left[ (D_{S(0)^{*}})^{-2} + \widetilde S(z) S(0)^{*} 
       (D_{S(0)^{*}})^{-2} \right] \cdot
       \left[ I + S(0) \widetilde S(w)^{*} \right] \\
       & \quad - \left[ S(0) (D_{S(0)})^{-2} + \widetilde S(z) 
       (D_{S(0)})^{-2} \right] \cdot
       \left[ S(0)^{*} + \widetilde S(w)^{*} \right]  \\
       & = (D_{S(0)^{*}})^{-2} - S(0) (D_{S(0)})^{-2} S(0)^{*}  \\
       & \quad + \widetilde S(z)  S(0)^{*} (D_{S(0)^{*}})^{-2} 
       S(0)\widetilde S(w)^{*} - \widetilde S(z) (D_{S(0)})^{-2} \widetilde S(w)^{*} + [\text{ 
       cross terms }]
       \end{align*}
       where we make use of \eqref{intertwine12} to see that the 
       cross terms vanish.  Continuation of the computation of $X$ 
       and again making use of \eqref{intertwine12} then gives:
       \begin{align*}
	   X & =  (D_{S(0)^{*}})^{-2} \left[ I - S(0) S(0)^{*}\right] + 
	   \widetilde S(z) \left[ S(0)^{*} S(0) - I\right] (D_{S(0)})^{-2} 
	   \widetilde S(w)^{*} \\
	   & = I - \widetilde S(z) \widetilde S(w)^{*}
 \end{align*}
 Plugging $X$ back into \eqref{Sdefect1} gives us \eqref{Sdefect} as 
 wanted.
\end{proof}

	With these preliminaries out of the way, we may use the 
        integral representation formula \eqref{Her-int-rep} for a 
	normalized Herglotz function to arrive at the following 
	representation for the defect kernel $I - S(z) S(w)^{*}$ for 
	a normalized Schur-class function $S$.
	To this end, we associate with any point $(\bx, \bw) \in 
	{\mathbb T}^{N}_{\cR}$ the normalized Schur-class function
	\begin{equation}   \label{extremal-Schur}
	    S_{\bx,\bw}(z) = (F_{\bx, \bw}(z) + I)^{-1} (F_{\bx, 
	    \bw}(z) - I)
	 \end{equation}
	 where $F_{\bx, \bw} \in \cH^{N}(\cR)_{I}$ is given by 
	 \eqref{extreme-Her}.

	\begin{theorem}  \label{T:Schur-def}
 Given $S$ in the strict Schur class $\cS^{N}(\cR)$, there is a 
 ${\mathbb C}^{N \times N}$-valued
 function $(z, (\bx, \bw))  \mapsto H_{\bx, \bw}(z)$ on ${\mathbb 
 T}^{N}_{\cR} \times \cR$,   with values bounded and 
measurable in $(\bx, \bw)$ for each fixed $z$ and 
 holomorphic in $z$ for each fixed $(\bx, \bw)$, along with a 
 probability measure  $\nu$ on 
 ${\mathbb T}^{N}_{\cR}$ so that
	    \begin{equation}  \label{Schur-def}
		I- S(z) S(w)^{*} =  
		\int_{{\mathbb T}^{N}_{\cR}} H_{\bx, \bw}(z) 
		\left( I - S_{\bx, \bw}(z) S_{\bx, \bw}(w)^{*} 
		\right) H_{\bx, \bw}(w)^{*}\, {\tt d}\nu(\bx, \bw).
	\end{equation}
	 \end{theorem}

	 \begin{proof}  We first consider the case where $S$ is in 
	     the normalized Schur class $\cS^{N}(\cR)_{0}$.  Then $F: 
	     = (I - S)^{-1} (I + S)$ is in the normalized Herglotz 
	     class $\cH^{N}(\cR)_{I}$ as explained in Proposition 
	     \ref{P:SchurHer}.  By Theorem \ref{T:Her-int-rep} there 
	     is a probability measure $\nu$ on ${\mathbb T}^{N}_{\cR}$ 
	     so that
	     $$
	     F(z) = \int_{{\mathbb T}^{N}_{\cR}} F_{\bx, \bw}(z) \, 
	     {\tt d}\nu(\bx, \bw).
	     $$
  If $S_{\bx, \bw}(z)$ is given by \eqref{extremal-Schur}, then we 
  know from Proposition \ref{P:SchurHer} that we recover $S_{\bx, 
  \bw}$ from $F_{\bx, \bw}$ according to
  $$
    S_{\bx, \bw}(z) = (F_{\bx, \bw}(z) + I)^{-1} (F_{\bx, \bw}(z) - 
    I).
  $$
  Then we compute
  \begin{align*}
     &  I - S(z) S(w)^{*}  = 2 (F(z) + I)^{-1} (F(z) + F(w)^{*} ) 
      (F(w)^{*} + I)^{-1} \text{( by \eqref{S-defect})}  \\
      & = \frac{1}{2} (I - S(z))  \int_{{\mathbb T}^{N}_{\cR}} 
      2 (I - S_{\bx, bw}(z))^{-1} (I - S_{\bx, \bw}(z) S_{\bx, 
      \bw}(w)^{*}) \, {\tt d}\nu(\bx, \bw)  (I - S(w)^{*}) \\
      & \quad \text{ (where we make use of \eqref{useful} and 
      \eqref{F-defect})}  \\
      & = \int_{{\mathbb T}^{N}_{\cR}} H_{\bx, \bw}(z) (I - S_{\bx, 
      \bw}(z) S_{\bx, \bw}(w)^{*}) H_{\bx, \bw}(w)^{*} \, {\tt 
      d}\nu(\bx, \bw)
      \end{align*}
    where we have set
    $$
    H_{\bx, \bw}(z) = (I - S(z)) (I - S_{\bx, \bw}(z))^{-1}.
    $$
    
    To handle the case where $S \in \cS^{N}(\cR)$ is not necessarily 
    normalized, we proceed as follows.  Write $S(z) = \left( L_{S(0)} 
    \right)^{-1} [\widetilde S(z) ]$ where $\widetilde S$ is in the 
    normalized Schur class $\cS^{N}(\cR)_{0}$.  Then, by the special 
    case of Theorem \ref{T:Schur-def} already proved, we know that 
    there is a probability measure $\nu$ and a function $\widetilde 
    H$ so that 
    $$
    I- \widetilde S(z) \widetilde S(w)^{*} =  
		\int_{{\mathbb T}^{N}_{\cR}} \widetilde H_{\bx, \bw}(z) 
		\left( I - S_{\bx, \bw}(z) S_{\bx, \bw}(w)^{*} 
		\right) \widetilde H_{\bx, \bw}(w)^{*}\, {\tt d}\nu(\bx, \bw).
  $$
  If we now use relation \eqref{Sdefect}, we see that 
  \eqref{Schur-def} holds for $S$ with
  $$
  H_{\bx, \bw}(z) = D_{S(0)^{*}} (I + \widetilde S(z) S(0)^{*})^{-1} 
  \widetilde H_{\bx, \bw}(z).
  $$
  
 \end{proof}
 
 Specializing this result to the scalar-valued case ($N=1$) recovers 
 the following result of Dritschel-McCullough.  To state the result 
 we introduce the scalar counterpart of the functions $S_{\bx, \bw}$ 
 given by \eqref{extremal-Schur}: for each point  $\bx$ in the 
 $\cR$-torus \eqref{Rtorus} 
 let $s_{\bx}$ be the scalar Schur-class function given by
 \begin{equation}  \label{extremal-Schur-scalar}
     s_{\bx}(z) = \frac{ f_{\bx}(z) + 1}{ f_{\bx}(z) - 1}
 \end{equation}
 where $f_{bx} \in \cH(\cR)_{1}$ is given by \eqref{fbx}.
 
 \begin{theorem}  \label{T:Schur-def-scalar}  (See \cite[Proposition 
     2.14]{DM}.)  Given a function $s$ on $\cR$ in the scalar-valued 
     Schur class $\cS(\cR)$, there are complex-valued functions 
     $(\bx, z) \mapsto h_{\bx}(z)$ on ${\mathbb T}_{\cR} \times \cR$, 
     bounded and measurable in $\bx$ for each fixed $z$ and 
     holomorphic in $z$ for each fixed $\bx$, and a positive 
     probability measure on ${\mathbb T}_{\cR}$, such that
     \begin{equation}   \label{Schur-def-scalar}
     1 - s(z) \overline{s(w)} = \int_{{\mathbb T}_{\cR}} h_{\bx}(z) 
     (1 - s_{\bx}(z) \overline{s_{\bx}(w)})\overline{h_{\bx}(w)}\, 
     {\tt d}\nu(\bx)
     \end{equation}
     
 \end{theorem}
 
 \begin{remark}  \label{R:minimal}
     {\em  We note that it is not possible to use a smaller closed 
     subset of ${\mathbb T}^{N}_{\cR}$ in the integral representation 
     \eqref{Her-int-rep} and still have the representation hold for 
     all $F \in \cH^{N}(\cR)_{I}$, almost by the definition of 
     extreme point.  However some reductions are always possible in 
     the decomposition \eqref{Schur-def}.  Note that we have already 
     imposed the normalization that $S_{\bx, \bw}(t_{0}) = 0$ for all 
     $(\bx, \bw)$.  In addition we note that the expression $I - 
     S_{\bx,\bw}(z) S_{\bx, \bw}(w)^{*}$ is unchanged if we replace 
     $S_{\bx,\bw}(z)$ by $S_{\bx, \bw}(z) U$ with $U$ a unitary $N 
     \times N$ matrix.   This means that we may restrict the integral 
     in \eqref{Schur-def} to points $(\bx, \bw)$ such that $S_{\bx 
     \bw}(\zeta_{0}) = I_{N}$ for some point $\zeta_{0}$ in $\partial 
     \cR$ (e.g., $\zeta_{0} \in \partial_{0}$) and consider the 
     integral over this smaller set $\widetilde {\mathbb T}^{N}_{\cR}$.  In 
     special situations for the $N=1$ case (see \cite{DM-Con} and 
     \cite{DP}), there are results proven that, after these 
     reductions, there is no proper closed subset 
     $\widetilde{\widetilde {\mathbb T}}_{\cR}$ of $\widetilde {\mathbb T}_{\cR}$ 
     for which a representation of the form \eqref{Schur-def-scalar} 
     can hold with ${\mathbb T}_{\cR}$ replaced by 
     $\widetilde{\widetilde {\mathbb T}}_{\cR}$.
     For the case $N > 1$, our description of the set ${\mathbb 
     T}^{N}_{\cR}$ (or of $\widetilde {\mathbb T}_{\cR}$) is not as 
     explicit as in the $N=1$ case, so as of this writing it is not at 
     all clear how to arrive at such minimality results for the 
     matrix-valued case.
    }\end{remark}
 
 \begin{remark}  \label{R:testfunc} {\em In \cite{DM} the authors go 
     on to use the general theory of the generalized Schur class 
     associated with a collection $\Psi$ of test functions (see 
     \cite{DMM, DM-Con}) to identify the Schur class $\cS(\cR)$ over 
     $\cR$ with the Schur class $\cS_{\Psi}$ associated with the 
     collection of test functions $\Psi = \{ s_{\bx} \colon \bx \in 
     {\mathbb T}_{\cR} \}$ and thereby also to obtain transfer-function 
     realizations for the class $\cS(\cR)$.  These results 
     combined with Theorem \ref{T:Schur-def} suggest that the 
     matrix-valued Schur class $\cS^{N}(\cR)$ is connected in a 
     similar way with the collection of matrix-valued test functions 
     ${\boldsymbol \Psi} = \{ S_{\bx, \bw} \colon (\bx, \bw) \in 
     {\mathbb T}^{N}_{\cR}\}$.  This is indeed the case (see 
     \cite{Mo, BG}).
     } \end{remark}

  \section{The spectral set problem}   \label{S:spectral}
  
  Let $\cR$ denote a domain in the complex plane ${\mathbb C}$ with 
  boundary $\partial \cR$ with closure $\cR^{-}$.  An operator $T$ on 
  a complex Hilbert space $\cH$ is said to have $\cR^{-}$ as a {\em 
  spectral set} if the spectrum $\sigma(T)$ of $T$ is contained in 
  $\cR^{-}$ and 
  $$
  \| f(T) \| \le \| f\|_{\cR} = \sup \{ |f(z)| \colon z \in \cR \}
  $$
  for every rational function $f$ with poles off of $\cR^{-}$, where 
  $f(T)$ can be defined by the Riesz functional calculus or simply as 
  $f(T) = p(T) q(T)^{-1}$ when $f$ is written as the ratio of 
  polynomials $f(z) = \frac{p(z)}{q(z)}$.  The operator $T$ is said 
  to have a $\partial \cR$-normal dilation if there exists a Hilbert 
  space $\cK$ containing $\cH$ as a subspace so that
  $$
    f(T) = P_{\cH} f(N) |_{\cH}
    $$
   for every rational function $f$ with poles off of $\cR^{-}$ (where 
   $P_{\cH}$ is the orthogonal projection of $\cK$ onto $\cH$.
   It is easily seen that if $T$ has a $\partial \cR$-normal 
   dilation, then $\cR^{-}$ is a spectral set for $T$.  The converse 
   question can be reformulated as:
   
   \smallskip
   \noindent 
   {\em Given that $\cR^{-}$ is a spectral set for $T$, 
   does it follow that $T$ has a $\partial \cR$-normal dilation?}
   
   \smallskip
   \noindent
   This has become known as the {\em spectral-set question} for $\cR$ (see 
   \cite{ArvII}).

   For the case of the unit disk $\cR = {\mathbb D}$, the von Neumann 
   inequality combined with the Sz.-Nagy dilation theorem implies a 
   positive answer to the spectral set question.  For the case where 
   $\cR = {\mathbb A}$ is an annulus, 
  it is a result of Agler \cite{Agler-annulus} (see also \cite{McC95}) that the spectral-set 
  question again has a positive answer.  However, for the case of a 
  multiply-connected domain with at least two holes,   more recent 
  work of Agler-Harland-Raphael \cite{AHR} and Dritschel-McCullough 
  \cite{DM} give two complementary approaches to showing that the 
  spectral-set question has a negative solution.  In this section we 
  discuss briefly how the ideas of this paper relate to the spectral 
  set question.
  
  In \cite{ArvII} Arveson obtained a reformulation of the spectral 
  set question which had profound influence on subsequent work.  For 
  our purposes it is convenient to assume that $\sigma(T)$ is 
  contained in the open domain $\cR$ rather than in $\cR^{-}$; in 
  this case we can use the standard Riesz holomorphic functional 
  calculus to define $s(T) \in \cL(\cH)$ for any holomorphic function 
  on $\cR$, in particular, for $s$ in the Schur class $\cS(\cR)$.  
  Then the condition that $T$ has $\cR^{-}$ as a spectral set can be 
  reformulated as: {\em for any $s \in \cS(\cR)$,  $\| s(T) \| \le 1$.}
  By the Arveson-Stinespring dilation theory (see \cite{Arv}), the 
  condition that $T$ have a $\partial \cR$-normal dilation can be 
  reformulated as: {\em for any $S \in \cS^{N}(\cR)$ ($N = 1,2, 
  \dots$),  $\|S(T) \| \le 1$.}
  Here, for $S = [s_{ij}]_{i,j=1}^{N}$ in the matrix-valued Schur 
  class $\cS^{N}(\cR)$, we define $S(T)$ by
  $$
     S(T) = [s_{ij}(T)] \in \cL(\cH^{N}).
   $$
   Then the Arveson reformulation of the spectral set question for 
   $\cR$ becomes:
   
   \smallskip 
   \noindent 
   {\em Given $T \in \cL(\cH)$ such that $\|s(T)\| \le 1$ for all $s 
   \in \cS(\cR)$, does it follow that $\|S(T) \| \le 1$ for all $S 
   \in \cS^{N}(\cR)$?}
   
   \smallskip
   \noindent
   
   From the result \eqref{Sdefect} of Proposition \ref{P:Sdefect}, we 
   see that in fact one may restrict to the normalized Schur classes 
   $\cS(\cR)_{0}$ and $\cS^{N}(\cR)_{0}$ in the above condition.
   We may then use relations \eqref{F-defect} and \eqref{S-defect} to 
   get the following reformulation:
   
   \smallskip
   \noindent
   {\em Given $T \in \cL(\cH)$ such that $\R f(T) \ge 0$ for all $f 
   \in \cH(\cR)_{1}$, does it follow that $\R F(T) \ge 0$ for all $F 
   \in \cH^{N}(\cR)_{I}$ for any $N = 1,2, \dots$ ?}
   
   \smallskip
   \noindent
   By plugging $T$ into the decomposition \eqref{Schur-def-scalar} with 
   the Riesz holomorphic functional calculus, we see that to check 
   whether $\cR^{-}$ is a spectral set for $T$, it suffices to check the 
   condition $\|s(T) \| \le 1$ only for $s=s_{\bx}$ for each $\bx \in 
   {\mathbb T}_{\cR}$. By using the relations \eqref{F-defect} and 
   \eqref{S-defect}, an equivalent condition for $\cR^{-}$ to be a 
   spectral set for $T$ is that $\R f_{\bx}(T) \ge 0$ for each $\bx 
   \in {\mathbb T}_{\cR}$ (see \cite[Theorem 1.6.16]{AHR}).  
  Similarly, to check that $T$ has a $\partial 
   \cR$-normal dilation, it suffices to check the condition $\|S(T)\| 
   \le 1$ only for $S \in \cS^{N}(\cR)_{0}$ of the form $S = S_{\bx, 
   \bw}$ for $(\bx, \bw) \in {\mathbb T}^{N}_{\cR}$.  By using 
   the relations \eqref{F-defect} and \eqref{Sdefect}, it is 
   equivalent to check that $\R F_{\bx, \bw}(T) \ge 0$ for each 
   $(\bx, \bw) \in {\mathbb T}^{N}_{\cR}$.
   We thus 
   come to the following equivalent reformulations of the spectral set question:
   
   \smallskip
   \noindent
   {\em Given $T \in \cL(\cH)$ for which $\|s_{\bx}(T)\| \le 1$ 
   (respectively, $\R f_{\bx}(T) \ge 0$) for all $\bx \in {\mathbb 
   T}_{\cR}$, does it follow that $\| S_{\bx, \bw}(T) \| \le 1$ 
   (respectively, $\R F_{\bx, \bw}(T) \ge 0$) for all $(\bx, \bw) \in 
   {\mathbb T}^{N}_{\cR}$?}
   
   \smallskip
   \noindent
   Let us suppose that all extremal measures $\mu$ for 
   the cone $\cC(\partial \cR, N, \bphi)$ are special (see Remark 
   \ref{R:special}).  Thus any $\mu_{\bx, \bw} \in \partial_{e} 
   \cC(\cR, N, \phi)$ has the form $\mu_{\bx, \bw} = \sum_{k=1}^{n} 
   \mu_{\bx_{k}} W_{k}$ for some points $\bx_{k} \in {\mathbb 
   T}_{\cR}$ and for  matrix weights $W_{k}$ such that 
   \{$\operatorname{Ran} W_{k} \colon 1 \le k \le n\}$ is a weakly 
   independent family of subspaces.  We then see that
   \begin{align*}
   F_{\bx, \bw}(z) + F_{\bx, \bw}(w)^{*} &  = \sum_{k=1}^{n} 
   (f_{\bx_{k}}(z) + \overline{f_{\bx_{k}}(w)}) W_k \\
   & = \sum_{k=1}^{n} W_k^{\frac{1}{2}} \left[ (f_{\bx_{k}}(z) + 
   \overline{f_{\bx_{k}}(w)}) I_{N} \right] W_k^{\frac{1}{2}}.
   \end{align*}
   The assumption that $\R f_{\bx}(T) \ge 0$ for each $\bx \in 
   {\mathbb T}_{\cR}$ then leads to the conclusion that $\R F_{\bx, 
   \bw}(T) \ge 0$ for each $(\bx, \bw) \in {\mathbb T}^{N}_{\cR}$, 
   and hence that the spectral-set question has an affirmative answer 
   for $\cR$, in contradiction with the results of \cite{AHR, DM}.
   These observations lead to the following corollary concerning the 
   structure of the set of extreme points for a cone  of the form 
   $\cC(\partial \cR, N, \bphi)$ with $\bphi$ as in \eqref{basis}.
   
   \begin{corollary}  \label{C:special}
       There are multiply-connected domains $\cR$ (with at 
       least two holes) such that not all extremal measures $\mu$ for 
       $\cC(\cR, N, \bphi)$ are special (as defined in 
       Remark \ref{R:special}).

       \end{corollary}
       
       \begin{remark}  \label{R:disk}  {\em
       We note that for the case where $\cR$ is the unit disk 
       ${\mathbb D}$, it is the case that all extremal measures are 
       special; hence the argument leading to Corollary \ref{C:special} yields 
       yet another proof that the spectral-set question for $\cR = {\mathbb D}$
       has an affirmative answer. 
       
       Alternatively, one could simply use the Herglotz integral 
       representation formula for the matrix-valued Herglotz class
       $\cH^{N}({\mathbb D})_{I}$ and work with quantum probability 
       measures, i.e., positive operator measures on $\mathbb T$ 
       with $\nu({\mathbb T}) = I_{N}$, rather than probability measures:
      \begin{equation}   \label{diskHer}
      F(z) = \int_{{\mathbb T}} \frac{ \zeta + z}{\zeta -z}\, {\tt 
      d}\nu(\zeta)
      \end{equation}
      The representation \eqref{scalarHer-int-rep} for the 
      scalar-valued Herglotz class over $\cR$ suggests a 
      representation for $F \in \cH^{N}(\cR)_{I}$ analogous to 
      \eqref{diskHer}:
      $$
      F(z) = \int_{{\mathbb T}_{\cR}} f_{\bx}(z) {\tt d}\nu(\bx)
      $$
      with $\nu$ a quantum probability measure.
      However the argument above leading to Corollary \ref{C:special} 
      shows that such a representation cannot possibly be true in 
      general when $\cR$ has at least two holes.
   }\end{remark}
   
   \begin{remark} \label{R:annulus}
   {\em We have not determined if all extremal 
       measures are special for the case of an annulus $\cR = 
       {\mathbb A}$.  However there is a somewhat different way to 
       reduce the matrix-valued Herglotz (or Schur) class to the 
       scalar Herglotz class for the case $\cR = {\mathbb A}$ due to 
       McCullough \cite{McC95} which we now describe.
       
       For ${\mathbb A}$ equal to the annulus ${\mathbb A}_{q} = \{ z 
       \in {\mathbb C} \colon q < |z| < 1\}$ (where $0<q<1$), it is 
       shown in \cite{McC95} that there is a curve $\{ \varphi_{\zeta} 
       \colon  \zeta \in {\mathbb T} \}$  of inner functions 
       over ${\mathbb A}_{q}$ (constructed from the Ahlfors function 
       $\varphi$ based at the point $\sqrt{q} \in {\mathbb A}_{q}$) 
       with the following special property.  First as a matter of 
       notation, for $t = (t_{1}, \dots, t_{N}) \in {\mathbb T}^{N}$, 
       let us let $\Phi_{t}(z)$ be the diagonal matrix inner function 
       over ${\mathbb A}_{q}$ given by
       \begin{equation}   \label{Phit}
       \Phi_{t}(z) = \begin{bmatrix} \varphi_{t_{1}}(z) & & \\ & \ddots 
       & \\ & & \varphi_{t_{N}}(z) \end{bmatrix} \text{ if } t = (t_{1}, 
       \dots, t_{N}),
       \end{equation}
       and, for $U$ a unitary $N \times N$ matrix and $t \in {\mathbb 
       T}^{N}$, let us set
       $$
       R_{U,t}(z) = (I_{N} + U \Phi_{t}(z)) ( I - U\Phi_{t}(z))^{-1}.
       $$
       Then one of the main results from \cite{McC95} is:  {\em given 
       a point $(\bx, \bw) \in {\mathbb T}^{N}_{{\mathbb A}_{q}}$ 
       with associated extremal Herglotz function $F_{\bx, \bw}$, 
       there is a $t \in {\mathbb T}^{N}$, an $N \times N$ unitary 
       matrix $U$, and an invertible $N \times N$ matrix $X$ so that
       \begin{equation} \label{Fxwrep}
	   \R F_{\bx, \bw}(z) = X(\R R_{U,t}(z)) X^{*}
	\end{equation}
	for all $z \in {\mathbb A}_{q}$.}
	Note that an easy computation gives 
	$$
	 R_{U,t}(z) + R_{U,t}(w)^{*} =
	2 (I - U \Phi_{t}(z))^{-1} U \left[ I - \Phi_{t}(z) 
	\Phi_{t}(w)^{*} \right] U^{*} (I - \Phi_{t}(w)^{*} 
	U^{*})^{-1}.
	$$
	
	Now suppose that $T \in \cL(\cK)$ with  $\sigma(T) \subset 
	{\mathbb A}_{q}$ has ${\mathbb A}_{q}$ as a spectral set, so 
	$\|s(T)\| \le 1$ for all $s $ in the scalar Schur class 
	$\cS({\mathbb A}_{q})$.  Then an immediate consequence of the 
	formula \eqref{Fxwrep} combined with the diagonal form 
	\eqref{Phit} of $\Phi_{t}$ is that it then follows that $\R 
	F_{\bx, \bw}(T) \ge 0$ as well, and thus the spectral-set 
	question has an affirmative answer for the annulus ${\mathbb 
	A}_{q}$.  This line of reasoning arguably provides some simplifications 
	to the solution of the	spectral set question for the annulus 
	given in \cite{McC95}.
        }\end{remark}

\end{document}